\documentclass[11pt]{article}
\usepackage{amsthm,amssymb,amsmath}
\usepackage{graphicx,algorithmic,algorithm}%
\usepackage{enumerate}
\usepackage{tikz}
\usetikzlibrary{shapes}
\usepackage[letterpaper, margin=1in]{geometry}
\usepackage{hyperref}
\usepackage{todonotes}
\usepackage[normalem]{ulem}
\hypersetup{
    pdftitle=   {Graphs without two disjoint $S$-cycles},
   pdfauthor=  {Minjeong Kang, O-joung Kwon, and Myounghwan Lee}
}
\usepackage{authblk}
\usepackage{mathabx}
\newtheorem{theorem}{Theorem}[section]
\newtheorem{lemma}[theorem]{Lemma}
\newtheorem{proposition}[theorem]{Proposition}

\newtheorem{claim}[theorem]{Claim}
\newtheorem{question}{Question}
\newtheorem*{thm:main1}{Theorem~\ref{thm:main1}}
\newtheorem*{thm:main2}{Theorem~\ref{thm:main2}}
\newenvironment{clproof}{\begin{list}{}{%
              \setlength{\leftmargin}{3mm}%
              } \item {\it Proof.} }{\hfill$\lozenge$\end{list}}

\newcommand\abs[1]{\lvert #1\rvert}

\newcommand{\dist}{\operatorname{dist}}

\begin{document}
\title{Graphs without two vertex-disjoint $S$-cycles}

\author{Minjeong Kang}
\author[1,2]{O-joung Kwon\thanks{
This work was supported by Incheon National University (International Cooperative) Research Grant in 2018.}}
\author{Myounghwan Lee}

\affil{Department of Mathematics, Incheon National University, Incheon, South Korea.}
\affil[2]{Discrete Mathematics Group, Institute for Basic Science (IBS), Daejeon, South Korea}

\date\today
\maketitle

\footnotetext{Corresponding author: O-joung Kwon. E-mail addresses: \texttt{minj2058@gmail.com} (M. Kang), \texttt{ojoungkwon@gmail.com} (O. Kwon), \texttt{sycuel2@gmail.com} (M. Lee) }

\begin{abstract}
Lov\'asz (1965) characterized graphs without two vertex-disjoint cycles, which implies that 
such graphs have at most three vertices hitting all cycles.
In this paper, we ask whether such a small hitting set exists for $S$-cycles, when a graph has no two vertex-disjoint $S$-cycles.
For a graph $G$ and a vertex set $S$ of $G$, an $S$-cycle is a cycle containing a vertex of $S$.

We provide an example $G$ on $21$ vertices where $G$ has no two vertex-disjoint $S$-cycles, but three vertices are not sufficient to hit all $S$-cycles.
On the other hand, we show that four vertices are enough to hit all $S$-cycles whenever a graph has no two vertex-disjoint $S$-cycles.

\end{abstract}

\section{Introduction}\label{sec:intro}

In this paper, we consider finite graphs that may have loops and multiple edges.
Erd\H{o}s and P\'osa~\cite{erdHos1962maximal} proved that for every graph $G$ and every positive integer $k$, 
$G$ contains either $k$ vertex-disjoint cycles or a vertex set of size $\mathcal{O}(k\log k)$ hitting all cycles.
This celebrated paper stimulated many researcher to find other classes that satisfy a similar property.
We say that a class $\mathcal{C}$ of graphs has the \emph{Erd\H{o}s-P\'osa property} 
if there is a function $f:\mathbb{Z}\rightarrow \mathbb{R}$ such that
for every graph $G$ and every integer $k$, $G$ contains either $k$ vertex-disjoint subgraphs each isomorphic to a graph in $\mathcal{C}$, 
or a vertex set of size at most $f(k)$ hitting all subgraphs isomorphic to a graph in $\mathcal{C}$.
We now know that several variations of cycles have this property: long cycles~\cite{BirmeleBR2007, BruhnJS2017, FioriniH2014, MoussetNSW2016, RobertsonS1986}, directed cycles~\cite{ReedRST1996, KakimuraK2012}, cycles with modularity constraints~\cite{HuyneJW2017, Thomassen1988}, holes~\cite{KK2018}, $S$-cycles~\cite{BruhnJS2017, KakimuraKM2011, PontecorviW2012}, and $(S_1, S_2)$-cycles~\cite{HuyneJW2017}. 
Sometimes, a variation of cycles does not have the Erd\H{o}s-P\'osa property. For instance, odd cycles do not have the Erd\H{o}s-P\'osa property~\cite{Reed1999}, 
and $(S_1, S_2, S_3)$-cycles do not have the Erd\H{o}s-P\'osa property~\cite{HuyneJW2017}. 
We refer to a recent survey of Erd\H{o}s-P\'osa property by Raymond and Thilikos~\cite{RaymondT16}.

For small values of $k$, we may ask to find the least possible value of $f(k)$.
For ordinary cycles, Bollob\'as (unpublished) first showed that when a graph has no two vertex-disjoint cycles, there are at most three vertices hitting all cycles, which gives a tight bound.
The complete graph $K_5$ has no two vertex-disjoint cycles but we need to take at least three vertices to hit all cycles.
Lov\'asz~\cite{Lovasz1965} characterized graphs without two vertex-disjoint cycles, which easily deduces that 
such graphs have at most three vertices hitting all cycles. 
Voss~\cite{Voss1969} showed that when a graph has no three vertex-disjoint cycles, there are at most six vertices hitting all cycles, which gives a tight bound.

For cycles of length at least $\ell$, 
Birmel{\'e}, Bondy, and Reed~\cite{BirmeleBR2007} conjectured that if a graph has no two vertex-disjoint cycles of length at least $\ell$, then there is a vertex set of size at most $\ell$ hitting all cycles of length at least $\ell$.
The complete graph $K_{2\ell-1}$ shows that this bound is tight if the conjecture is true. 
Birmel{\'e}~\cite{Bermele2003} confirmed that this conjecture is true for $\ell=4, 5$, but it remains open  for $\ell\ge 6$.

In this paper, we determine the tight possible value $f(k)$ for $S$-cycles, when $k=2$.
A pair $(G,S)$ of a graph $G$ and its vertex set $S$ is called a \emph{rooted graph}.
For a rooted graph $(G, S)$, a cycle of $G$ that contains a vertex of $S$ is called an \emph{$S$-cycle}. 
A vertex set $T$ of $G$ is called an \emph{$S$-cycle hitting set} if $T$ meets all the $S$-cycles of $G$.
We denote by $\mu(G, S)$ the maximum number of vertex-disjoint $S$-cycles in $(G, S)$, and 
denote by $\tau(G, S)$ the minimum size of an $S$-cycle hitting set in $(G, S)$.

As we listed above, it is known that $S$-cycles~\cite{KakimuraKM2011, PontecorviW2012, BruhnJS2017} have the Erd\H{o}s-P\'osa property.
$S$-cycles have an important role in studying the \textsc{Subset Feedback Vertex Set} problem~\cite{EvenNZ2000, CyganPPW2013}, which asks whether a given graph has at most $k$ vertices hitting all $S$-cycles.

Compared to the ordinary cycles, we find an example $(G, S)$ such that $\mu(G, S)= 1$ but $\tau(G, S)=4$.
This example is illustrated in Figure~\ref{fig:example}, and we devote Section~\ref{sec:lower} to prove it.

\begin{theorem}\label{thm:main2}
There is a rooted graph $(G, S)$ on $21$ vertices such that $\mu(G, S)=1$ and $\tau(G, S)\ge 4$.
\end{theorem}

On the other hand, we show that four vertices are always enough to hit all $S$-cycles whenever $\mu(G, S)\le 1$.
So, we determine the tight bound for the $S$-cycles.

\begin{theorem}\label{thm:main1}
	Let $(G, S)$ be a rooted graph. If $\mu(G, S)\le 1$, then $\tau(G, S)\le 4$.
\end{theorem}	

We explain a strategy for Theorem~\ref{thm:main1}.
We will say that a subgraph $H$ is an \emph{$S$-cycle subgraph}, if every cycle of $H$ is an $S$-cycle.
Suppose that $\mu(G, S)\le 1$ and $\tau(G, S)> 4$, and we will obtain a contradiction.
As $\tau(G, S)>4$, $G$ has an $S$-cycle which is an $S$-cycle subgraph.
Starting from this $S$-cycle, we recursively find a larger $H$-subdivision for some $H$, which is an $S$-cycle subgraph.

Assume that $W$ is an $S$-cycle $H$-subdivision for some $H$.
To find a larger $S$-cycle subgraph, we want to find a $W$-path $X$ where the union of $W$ and $X$ is again an $S$-cycle subgraph.
Clearly, not all $W$-paths can be added. 
A way to guarantee the existence of such a $W$-path is the following. Let $T$ be a subset of $S\cap V(W)$ which is an $S$-cycle hitting set of $W$.
Then we show that if $G$ still has an $S$-cycle $C$ that does not meet $T$, then either $C$ attaches to $W$ on one vertex, or $C$ contains a $W$-path $X$ such that 
the union of $W$ and $X$ is again an $S$-cycle subgraph (Lemma~\ref{lem:intersection2}). So, if such a small hitting set $T$ exists, then we can easily find a larger $S$-cycle subgraph.
We simply apply this argument whenever the current subgraph $W$ admits an $S$-cycle hitting set of size at most $4$ that is contained in $S$.

By this approach, we end up with some structures $W$ where 
we cannot simply guarantee the existence of an $S$-cycle hitting set of size at most $4$ contained in $S$.
For those subgraphs, we will analyze their structures and show that the existence of such a structure would imply that 
$G$ has an $S$-cycle hitting set of size at most $4$. This will complete the proof.

The paper is organized as follows.
We introduce basic notions  in Section~\ref{sec:prelim}.
In Section~\ref{sec:lower}, we prove Theorem~\ref{thm:main2}.
In Section~\ref{sec:overview}, we give a detailed overview of Theorem~\ref{thm:main1} 
with introducing additional notions and basic lemmas.
See Subsection~\ref{subsec:specialgraph} and Figure~\ref{fig:specialgraph} for definitions of special graphs.
\begin{itemize}
	\item In Section~\ref{sec:k4subdiv}, 
	 we prove that either $G$ contains an $S$-cycle $K_4$-subdivision or it has an $S$-cycle hitting set of size at most $4$.
	 \item In Section~\ref{sec:w4ork33}, 
	 we prove that either $G$ contains an $S$-cycle $W_4$-subdivision or it has an $S$-cycle $K_{3,3}^+$-subdivision or it has an $S$-cycle hitting set of size at most $4$.	 
	 \item In Section~\ref{sec:varw4},
	 we prove that when $G$ contains an $S$-cycle $W_4$-subdivision, 
	 either $G$ has an $S$-cycle $K_{3,3}^+$-subdivision or it has an $S$-cycle hitting set of size at most $4$.
	 \item In Section~\ref{sec:k33final},
	 we prove that when $G$ contains an $S$-cycle $K_{3,3}^+$-subdivision,
	 $G$ has an $S$-cycle hitting set of size at most $4$. 
\end{itemize}

\section{Preliminaries}\label{sec:prelim}

For a graph $G$, we denote by $V(G)$ and $E(G)$ the vertex set and the edge set of $G$, respectively.
Let $G$ be a graph. 
For a vertex set $S$ of $G$, let $G[S]$ denote the subgraph of $G$ induced by $S$, and 
let $G-S$ denote the subgraph of $G$ obtained by removing all the vertices in $S$.
For $v\in V(G)$, let $G-v:=G-\{v\}$.
Similarly, for an edge set $F$ of $G$, let $G-F$ denote the subgraph of $G$ obtained by removing all the edges in $F$, 
and for $e\in E(G)$, let $G-e:=G-\{e\}$.
If two vertices $u$ and $v$ are adjacent in $G$, then we say that $u$ is a \emph{neighbor} of $v$. 
The set of neighbors of a vertex $v$ is denoted by $N_G(v)$, and the \emph{degree} of $v$ is defined as the size of $N_G(v)$.
For two vertices $v$ and $w$ in $G$, we denote by $\dist_G(v,w)$ the distance between $v$ and $w$ in $G$; that is, the length of a shortest path from $v$ to $w$ in $G$.
Two subgraphs $H$ and $F$ of $G$ are \emph{vertex-disjoint}, or \emph{disjoint} for short, if $V(H)\cap V(F)=\emptyset$. 

For two graphs $G$ and $H$, $G\cup H$ denotes the graph $(V(G)\cup V(H), E(G)\cup E(H))$.

A \emph{subdivision of $H$} (\emph{$H$-subdivision} for short) is a graph obtained from $H$ by subdividing some of its edges.
For an $H$-subdivision $W$, the vertices of $H$ in $W$ are called the \emph{branching vertices} of $W$, 
and a path between two branching vertices that contains no other branching vertex is called a \emph{certifying path} of $W$.

\subsection{Rooted graphs}
A pair $(G, S)$ of a graph $G$ and its vertex subset $S$ is called a \emph{rooted graph}.
An \emph{$S$-cycle subgraph} of a rooted graph $(G, S)$ is a subgraph whose every cycle is an $S$-cycle.
In particular, an \emph{$S$-cycle $H$-subdivision} is an $H$-subdivision whose every cycle is an $S$-cycle.

For a subgraph $H$ of $G$, a path $P$ with at least one edge is called an \emph{$H$-path} if its endpoints are contained in $H$ but all the other vertices are not in $H$, and it is not an edge of $H$.
For two subgraphs $F_1$ and $F_2$ of $H$, an $H$-path is called an \emph{$(H,F_1, F_2)$-path} if its one endpoint is contained in $F_1$ and the other endpoint is contained in $F_2$.

Given an $S$-cycle subgraph $W$, we say that a $W$-path $X$ is a \emph{$W$-extension} if 
$W\cup X$ is again an $S$-cycle subgraph.

\subsection{Special graphs}\label{subsec:specialgraph}

\begin{figure}
 \tikzstyle{v}=[circle, draw, solid, fill=black, inner sep=0pt, minimum width=3pt]
 \tikzstyle{w}=[rectangle, draw, solid, fill=black, inner sep=0pt, minimum width=5pt, minimum height=5pt]
  \centering
   \begin{tikzpicture}[scale=0.6]

        \node [v]  (z1) at (0, 0){};
        \node [v]  (z2) at (2, 0){};
        \node [v]  (z3) at (1, 1.7){};
        
       \draw(z1)--(z2)--(z3)--(z1);
      \draw(z1) [in=180,out=120] to (z3);
      \draw(z1) [in=200,out=90] to (z3);
       \draw(z2) [in=0,out=60] to (z3);
      \draw(z1) [in=-120,out=-60] to (z2);
      \draw (1,-1) node [below] {$K_3^{+++}$};

        \node [v]  (v1) at (0-4, 0){};
        \node [v]  (v2) at (2-4, 0){};
        \node [v]  (v3) at (1-4, 1.7){};
        
       \draw(v1)--(v2)--(v3)--(v1);
      \draw(v1) [in=180,out=120] to (v3);
       \draw(v2) [in=0,out=60] to (v3);
      \draw(v1) [in=-120,out=-60] to (v2);
      \draw (1-4,-1) node [below] {$K_3^{++}$}; 
      
        \node [v]  (w1) at (0-4-4, 0){};
        \node [v]  (w2) at (2-4-4, 0){};
        \node [v]  (w3) at (1-4-4, 1.7){};
        
       \draw(w1)--(w2)--(w3)--(w1);
      \draw(w1) [in=180,out=120] to (w3);
       \draw(w2) [in=0,out=60] to (w3);
       \draw (1-4-4,-1) node [below] {$K_3^+$}; 

         \node [v]  (a1) at (0-8, 0-10){};
        \node [v]  (a2) at (2-8, 0-10){};
        \node [v]  (a3) at (2-8, 2-10){};
        \node [v]  (a4) at (0-8, 2-10){};
            \node [v]  (a5) at (1-8, 1-10){};
    	\draw(a1)--(a2)--(a3)--(a4)--(a1);
	\draw(a5)--(a1);    
	\draw(a5)--(a2);    
 	\draw(a5)--(a3);    
 	\draw(a5)--(a4);    
         \draw (1-8,-1-10) node [below] {$W_4^+$}; 
            \draw(a5) [in=-22,out=110] to (a4);
    
        \node [v]  (b1) at (0, 0-10){};
        \node [v]  (b2) at (0, 1-10){};
        \node [v]  (b3) at (0, 2-10){};
        \node [v]  (b4) at (2, 0-10){};
        \node [v]  (b5) at (2, 1-10){};
        \node [v]  (b6) at (2, 2-10){};
        
        \draw(b1)--(b4);
        \draw(b1)--(b5);
        \draw(b1)--(b6);
        \draw(b2)--(b4);
        \draw(b2)--(b5);
        \draw(b2)--(b6);
        \draw(b3)--(b4);
        \draw(b3)--(b5);
        \draw(b3)--(b6);
        \draw(b2)--(b3);
        \draw (1,-1-10) node [below] {$K_{3, 3}^+$};

		\node [v]  (c1) at (0, 0-5){};
        \node [v]  (c2) at (2, 0-5){};
        \node [v]  (c3) at (1, 1.7-5){};
        \node [v]  (c4) at (1, 0.8-5){};
	
		\draw(c1)--(c2)--(c3)--(c1);
		\draw(c4)--(c1);        
		\draw(c4)--(c2);        
		\draw(c4)--(c3);        
        \draw (1,-1-5) node [below] {$K_4^{+++}$}; 
          \draw(c1) [in=180,out=120] to (c3);
          \draw(c1) [in=200,out=90] to (c3);

\node [v]  (c1) at (0-4, 0-5){};
        \node [v]  (c2) at (2-4, 0-5){};
        \node [v]  (c3) at (1-4, 1.7-5){};
        \node [v]  (c4) at (1-4, 0.8-5){};
	
		\draw(c1)--(c2)--(c3)--(c1);
		\draw(c4)--(c1);        
		\draw(c4)--(c2);        
		\draw(c4)--(c3);        
        \draw (1-4,-1-5) node [below] {$K_4^{++}$}; 
          \draw(c1) [in=180,out=120] to (c3);
		\draw(c2) [in=0,out=60] to (c3);

        \node [v]  (c1) at (0-8, 0-5){};
        \node [v]  (c2) at (2-8, 0-5){};
        \node [v]  (c3) at (1-8, 1.7-5){};
        \node [v]  (c4) at (1-8, 0.8-5){};
	
		\draw(c1)--(c2)--(c3)--(c1);
		\draw(c4)--(c1);        
		\draw(c4)--(c2);        
		\draw(c4)--(c3);        
        \draw (1-8,-1-5) node [below] {$K_4^+$}; 
          \draw(c1) [in=180,out=120] to (c3);

        \node [v]  (b1) at (0-4, 0-10){};
        \node [v]  (b2) at (2-4, 0-10){};
        \node [v]  (b3) at (2-4, 2-10){};
        \node [v]  (b4) at (0-4, 2-10){};
        \node [v]  (b5) at (1-4, 1-10){};
        
        \draw(b1)--(b2)--(b3)--(b4)--(b1);
        \draw(b5)--(b1);
        \draw(b5)--(b2);
        \draw(b5)--(b3);
        \draw(b5)--(b4);
         \draw (1-4,-1-10) node [below] {$W_4^*$}; 
       \draw(b2) [in=-22,out=110] to (b4);
    
\end{tikzpicture}
   \caption{Graphs that appear in the proof.}
  \label{fig:specialgraph}
\end{figure}
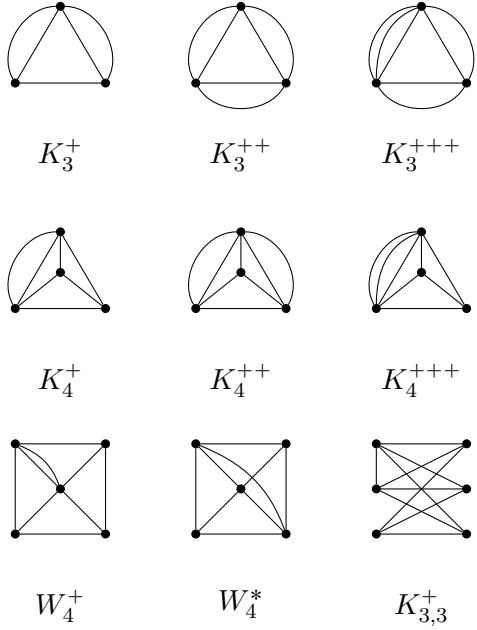

Let $m$ and $n$ be positive integers.
Let $K_n$ be the complete graph on $n$ vertices, 
and let $K_{m,n}$ be the complete bipartite graph where one part has size $m$ and the other part has size $n$.
Let $\theta_n$ denote the graph consisting of two vertices with $n$ multiple edges between them.
For $n\ge 3$, let $W_n$ denote the graph obtained from a cycle on $n$ vertices by adding a vertex adjacent to all the vertices of the cycle.

We define special graphs illustrated in Figure~\ref{fig:specialgraph}.
\begin{itemize}
\item Let $K_3^{+}$ denote the graph obtained from $K_3$ by adding a multiple edge to each of two distinct edges of $K_3$.
Let $K_3^{++}$ denote the graph obtained from $K_3$ by adding a multiple edge to each edge of $K_3$.
Let $K_3^{+++}$ denote the graph obtained from $K_3$ by adding a multiple edge to each edge of $K_3$ and then adding one more multiple edge to an edge of the resulting graph.
\item Let $K_4^{+}$ denote the graph obtained from $K_4$ by adding a multiple edge to an edge of $K_4$.
Let $K_4^{++}$ denote the graph obtained from $K_4$ by adding a multiple edge to each of two incident edges of $K_4$.
Let $K_4^{+++}$ denote the graph obtained from $K_4$ by adding two multiple edges to an edge of $K_4$.
\item Let $W_4^{+}$ denote the graph obtained from $W_4$ by adding an edge between a vertex of degree $4$ and a vertex of degree $3$.
 Let $W_4^{*}$ denote the graph obtained from $W_4$ by adding an edge between two vertices of distance $2$.
\item Let $K_{3,3}^+$ denote the graph obtained from $K_{3,3}$ by adding an edge between two vertices in the same part.
\end{itemize}

\section{An example showing that three vertices are not sufficient}\label{sec:lower}

In this section, we prove Theorem~\ref{thm:main2}.

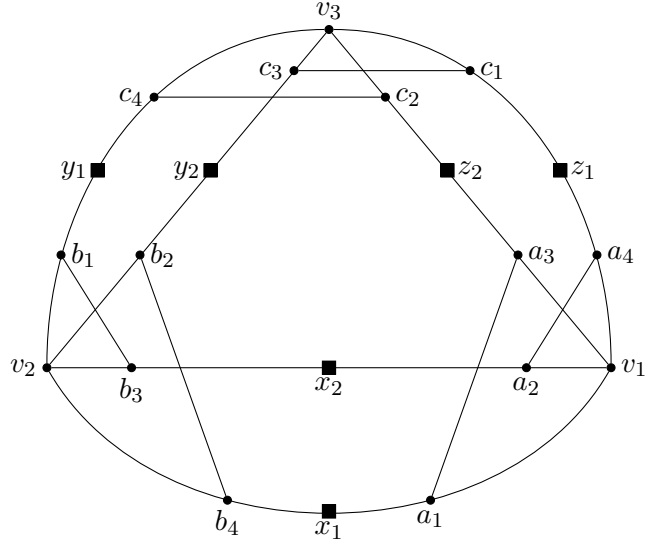
\begin{figure}
 \tikzstyle{v}=[circle, draw, solid, fill=black, inner sep=0pt, minimum width=3pt]
 \tikzstyle{w}=[rectangle, draw, solid, fill=black, inner sep=0pt, minimum width=5pt, minimum height=5pt]
  \centering
   \begin{tikzpicture}[scale=0.75]
      
        \node [v]  (v1) at (-5, 0){};
        \node [v]  (v2) at (0, 6){};
        \node [v]  (v3) at (5, 0){};
      
       \draw(v1)--(v2)--(v3)--(v1);
        \draw (v1) node [left] {$v_2$}; 
         \draw (v2) node [above] {$v_3$}; 
         \draw (v3) node [right] {$v_1$};

      \draw(v1) [in=180,out=90] to (v2);
      \draw(v3) [in=0,out=90] to (v2);
      \draw(v1) [in=-120,out=-60] to (v3);
      
        \node [v]  (a1) at (4.75, 2){};
        \node [v]  (b1) at (-4.75, 2){};
        \node [v]  (a2) at (3.35, 2){};
        \node [v]  (b2) at (-3.35, 2){};
        \node [v]  (a3) at (3.5, 0){};
        \node [v]  (b3) at (-3.5, 0){};
        \node [v]  (a4) at (1.8, -2.35){};
        \node [v]  (b4) at (-1.8, -2.35){};

        \draw (a1) node [right] {$a_4$}; 
        \draw (a2) node [right] {$a_3$}; 
        \draw (a3) node [below] {$a_2$}; 
        \draw (a4) node [below] {$a_1$}; 
        \draw (b1) node [right] {$b_1$}; 
        \draw (b2) node [right] {$b_2$}; 
        \draw (b3) node [below] {$b_3$}; 
        \draw (b4) node [below] {$b_4$}; 
      
      \draw(a1)--(a3); \draw(a2)--(a4);
      \draw(b1)--(b3); \draw(b2)--(b4);
      
       \node [v]  (c1) at (2.5, 5.27){};
       \node [v]  (d1) at (-3.1, 4.8){};
       \node [v]  (c2) at (1, 4.8){};
       \node [v]  (d2) at (-0.62, 5.27){};
       
         \draw (c1) node [right] {$c_1$}; 
        \draw (c2) node [right] {$c_2$}; 
        \draw (d1) node [left] {$c_4$}; 
        \draw (d2) node [left] {$c_3$}; 
      
       \draw(c1)--(d2);
       \draw(c2)--(d1);

       \node [w]  (e1) at (4.1, 3.5){};
       \node [w]  (f1) at (-4.1, 3.5){};
       \node [w]  (e2) at (2.1, 3.5){};
       \node [w]  (f2) at (-2.1, 3.5){};
       \node [w]  (e3) at (0, 0){};
       \node [w]  (f3) at (0, -2.55){};
      
       \draw (e1) node [right] {$z_1$}; 
        \draw (e2) node [right] {$z_2$}; 
        \draw (e3) node [below] {$x_2$}; 
        \draw (f1) node [left] {$y_1$}; 
        \draw (f2) node [left] {$y_2$}; 
        \draw (f3) node [below] {$x_1$};

\end{tikzpicture}
   \caption{A graph $G$ with $S=\{x_1, x_2, y_1, y_2, z_1, z_2\}$ where $G$ has no two vertex-disjoint $S$-cycles, but it has no $S$-cycle hitting set of size at most $3$.}
  \label{fig:example}
\end{figure}

\begin{thm:main2}
There is a rooted graph $(G, S)$ on $21$ vertices such that $\mu(G, S)=1$ and $\tau(G, S)\ge 4$.
\end{thm:main2}
\begin{proof}
	We define a graph $G$ with a vertex set $S=\{x_1, x_2, y_1, y_2, z_1, z_2\}$ as illustrated in Figure~\ref{fig:example}.
	
	We first show that $G$ has no two vertex-disjoint $S$-cycles.
	Suppose that $G$ has two vertex-disjoint $S$-cycles  $C_1$ and $C_2$.
	By symmetry, we may assume that $C_1$ contains $x_1$.
	Note that $C_1$ contains one of $b_4b_2$ and $b_4v_2$, 
	and similarly, it contains one of $a_1a_3$ and $a_1v_1$.
	
	In each case, we can observe that $C_2$ contains neither $y_2$ nor $z_2$. As $C_2$ contains $y_1$ or $z_1$, 
	we can see that $C_2$ contains a path from $c_4$ to $c_1$ in 
	$G[\{c_4, y_1, b_1, b_3, v_2, x_2, a_2, a_4, v_1, z_1, c_1\}]$. Then $C_2$ has to contain $v_3$ as well.  
	It implies that $C_1-x_1$ cannot connect the part $\{b_4, v_2, b_2, y_2, c_3\}$ and the part $\{a_1, v_1, a_3, z_2, c_2\}$, a contradiction.
		It shows that $G$ has no two vertex-disjoint $S$-cycles.

		Now, we prove that $G$ has no vertex set of size at most $3$ hitting all $S$-cycles.
	Suppose $T$ is a vertex set of size at most $3$ hitting all $S$-cycles such that $\abs{T}$ is minimum.
	As each vertex in $\{x_1, x_2, y_1, y_2, z_1, z_2\}$ has degree $2$, 
	we may assume that $T\cap \{x_1, x_2, y_1, y_2, z_1, z_2\}=\emptyset$.
	
	If $T$ contains no vertex in $\{a_i, b_i, c_i:1\le i\le 4\}$, 
	then $T$ does not meet the $S$-cycle 
	\[ a_1x_1b_4b_2y_2c_3c_1z_1a_4a_2x_2b_3b_1y_1c_4c_2z_2a_3a_1. \]
	So, $T$ contains a vertex in $\{a_i, b_i, c_i:1\le i\le 4\}$,
	and by symmetry, we may assume that $a_1\in T$.
 	Observe that any $S$-cycle in $G-a_1$ does not contain $x_1$.
	Furthermore, in $G-\{a_1, x_1\}$, $b_4$ has degree $2$, and its neighbors are adjacent.
	So, any $S$-cycle containing $b_4$ can be shorten using the edge $b_2v_2$.
	By the minimality of $T$, we have that $T\setminus \{a_1\}\subseteq V(G)\setminus \{a_1, x_1, b_4\}$
	and $T\setminus \{a_1\}$ intersects all $S$-cycles in $G-\{a_1, x_1, b_4\}$. 

	Let $A=\{v_1, a_2, a_3, a_4\}$, $B=\{v_2, b_1, b_2, b_3\}$, and $C=\{v_3, c_1, c_2, c_3, c_4\}$.
	We claim that $T\setminus \{a_1\}$ is not contained in any of $A, B$, and $C$. 
	If $T\setminus \{a_1\}\subseteq A$, then $T$ does not meet $v_2b_2y_2c_3v_3c_4y_1b_1v_2$, 
	and similarly, if $T\setminus \{a_1\}\subseteq B$, then  $T$ does not meet $v_3c_1z_1a_4v_1a_3z_2c_2v_3$.
	Assume that $T\setminus \{a_1\}\subseteq C$. 
    If $T\setminus \{a_1\}\subseteq \{c_1, c_2\}$, then 
    $T$ does not meet $v_2b_2y_2c_3v_3c_4y_1b_1v_2$.
    Also, if $T\setminus \{a_1\}\subseteq \{c_3, c_4\}$, then 
    $T$ does not meet $v_3c_1z_1a_4v_1a_3z_2c_2v_3$.
    So, we may assume that $T\setminus \{a_1\}$ is contained in neither $\{c_1, c_2\}$ nor $\{c_3, c_4\}$.
	Then there is a path from $\{y_1, y_2\}$ to $\{z_1, z_2\}$
	in $(G-T)[\{y_1, y_2, z_1, z_2\}\cup C]$.
	Since there is a path from $y_i$ to $z_j$ for any pair of $i,j\in \{1,2\}$ in $(G-T)[A\cup B\cup \{x_2, y_1, y_2, z_1, z_2\}]$,
	we can find an $S$-cycle avoiding $T\setminus \{a_1\}$, which is a contradiction.
	This shows that $T\setminus \{a_1\}$ is not contained in any of $A, B$, and $C$, 
	and it implies that $T\setminus \{a_1\}$ consists of two vertices from distinct sets of $A, B$, and $C$. 
	 
	 As $A$ and $B$ are symmetric in $G-\{a_1, x_1, b_4\}$, 
	 we may assume that 
	 either 
	 \begin{itemize}
		\item $\abs{T\cap A}=1$ and $\abs{T\cap B}=1$, or
	 	\item $\abs{T\cap A}=1$ and $\abs{T\cap C}=1$.
	 \end{itemize}
	 We divide into those cases, and for each case, we show that $G-T$ contains an $S$-cycle, leading a contradiction.
	 
	 \begin{itemize}
	 	\item (Case 1. $\abs{T\cap A}=1$ and $\abs{T\cap B}=1$.)\\
		We show that there is a path from $\{y_1, y_2\}$ to $\{z_1, z_2\}$ in $G[A\cup B\cup \{x_2, y_1, y_2, z_1, z_2\}]$ avoiding $T$. 
		If $T\cap A=\{a_2\}$, then $v_3c_1z_1a_4v_1a_3z_2c_2v_3$ is still an $S$-cycle in $G-T$.
		So, $a_2$ is not in $T$, and 
		there is a path from $x_2$ to $\{z_1, z_2\}$ in $G[A\cup \{x_2, z_1, z_2\}]$ avoiding $T$.
		Similarly,
		$b_3$ is not in $T$,  and
		there is a path from $x_2$ to $\{y_1, y_2\}$ in $G[B\cup \{x_2, y_1, y_2\}]$ avoiding $T$.
		These imply that 
		there is a path from $\{y_1, y_2\}$ to $\{z_1, z_2\}$ in $G[A\cup B\cup \{x_2, y_1, y_2, z_1, z_2\}]$ avoiding $T$, 
		and we can find an $S$-cycle by connecting through $C$.
		\item (Case 2. $\abs{T\cap A}=1$ and $\abs{T\cap C}=1$.) \\
	 	We show that there is a path from $x_2$ to $\{y_1, y_2\}$ in $G[A\cup C\cup \{x_2, y_1, y_2, z_1, z_2\}]$ avoiding $T$. 
		If the vertex of $T\cap C$ is $c_1$ or $c_2$, 
		then $T$ does not meet the $S$-cycle $v_2b_2y_2c_3v_3c_4y_1b_1v_2$.
		So, we may assume that $T\cap C\subseteq \{v_3, c_3, c_4\}$.
		
		First assume that $T\cap C=\{v_3\}$.
		If $T\cap A=\{a_2\}$, then 
		\[v_2b_2y_2c_3c_1z_1a_4v_1a_3z_2c_2c_4y_1b_1v_2\] is  a remaining $S$-cycle, a contradiction.
		So, the vertex of $T\cap A$ is contained in $\{v_1, a_3, a_4\}$.
		Then there is a path from $x_2$ to $\{z_1, z_2\}$ in $G[A\cup C\cup \{x_2, y_1, y_2, z_1, z_2\}]$ avoiding $T$, 
		and we can connect to $\{y_1, y_2\}$ along $c_1c_3y_2$ or $c_2c_4y_1$.
		
		Secondly, assume that $T\cap C=\{c_3\}$. 
		If $T\cap A=\{a_2\}$, then 
		\[v_3c_2z_2a_3v_1a_4z_1c_1v_3\] is  a remaining $S$-cycle, a contradiction.
		So, the vertex of $T\cap A$ is contained in $\{v_1, a_3, a_4\}$.
	Then there is a path from $x_2$ to $\{z_1, z_2\}$ in $G[A\cup C\cup \{x_2, y_1, y_2, z_1, z_2\}]$ avoiding $T$, 
		and we can connect to $y_1$ along $c_1v_3c_4y_1$ or $c_2c_4y_1$.
		
		Lastly, assume that $T\cap C=\{c_4\}$. 
		If $T\cap A=\{a_2\}$, then 
		\[v_3c_2z_2a_3v_1a_4z_1c_1v_3\] is  a remaining $S$-cycle, a contradiction.
		So, the vertex of $T\cap A$ is contained in $\{v_1, a_3, a_4\}$.
	Then there is a path from $x_2$ to $\{z_1, z_2\}$ in $G[A\cup C\cup \{x_2, y_1, y_2, z_1, z_2\}]$ avoiding $T$, 
				and we can connect to $y_2$ along $c_1c_3y_2$ or $c_2v_3c_3y_2$.
				
				Thus, we can find an $S$-cycle connecting through $B$.
			 \end{itemize}
	 We conclude that $G-T$ contains an $S$-cycle. It contradicts our assumption that $T$ is an $S$-cycle hitting set.
\end{proof}

\section{Basic lemmas for Theorem~\ref{thm:main1}}\label{sec:overview}

In this section, 
we introduce some necessary notions and prove basic lemmas regarding $S$-cycle $H$-subdivisions. 

\subsection{Paths with specified vertices}\label{sec:paths}

Let $(G, S)$ be a rooted graph and $P$ be a path with endpoints $v$ and $w$.
We define $P_{mid}$ as the shortest subpath of $P-\{v, w\}$ containing all the vertices of $(V(P)\setminus \{v, w\})\cap S$.
If $P-\{v, w\}$ contains no vertex of $S$, then it is defined to be the empty graph.
The endpoints of $P_{mid}$ will be called the \emph{gates} of $P$.
For an endpoint $z$ of $P$, the component of $P-V(P_{mid})$ containing $z$ is denoted by $P_z$.
See Figure~\ref{fig:pmid} for an illustration.

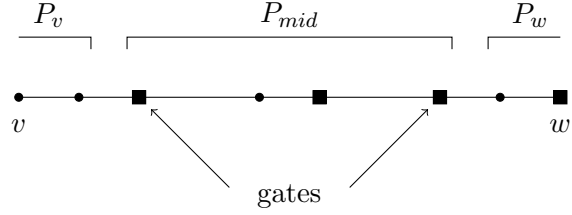
\begin{figure}
 \tikzstyle{v}=[circle, draw, solid, fill=black, inner sep=0pt, minimum width=3pt]
 \tikzstyle{w}=[rectangle, draw, solid, fill=black, inner sep=0pt, minimum width=5pt, minimum height=5pt]
  \centering
   \begin{tikzpicture}[scale=0.8]
      
      \draw(1, 0)--(10, 0);

        \node [v]  (v) at (1, 0){};
      	\node [w]  (v) at (10, 0){};

        \node [v]  (v1) at (2, 0){};
        \node [v]  (v2) at (5, 0){};
        \node [w]  (w1) at (3, 0){};
        \node [w]  (w2) at (6, 0){};
        \node [w]  (w3) at (8, 0){};
        \node [v]  (v4) at (9, 0){};
        
        \draw (2.8,0.8)--(2.8,1)--(8.2,1)--(8.2,0.8);
        \draw (2.2,0.8)--(2.2,1)--(1,1);
        \draw (8.8,0.8)--(8.8,1)--(10,1);
        
        \draw (5.5,1) node [above] {$P_{mid}$}; 
        \draw (1.5,1) node [above] {$P_v$}; 
        \draw (9.5,1) node [above] {$P_w$}; 
        \draw (1,-.2) node [below] {$v$}; 
        \draw (10,-.2) node [below] {$w$}; 

	    \draw (5.5,-2) node [above] {gates}; 
	    \draw[->] (4.5, -1.5)--(3.2,-.2);
		\draw[->] (6.5, -1.5)--(7.8,-.2);

\end{tikzpicture}
   \caption{A path $P$ in a rooted graph $(G,S)$. Rectangles depict vertices in $S$. }
  \label{fig:pmid}
\end{figure}

We frequently use the following lemma.

\begin{lemma}\label{lem:mid}
Let $(G, S)$ be a rooted graph with $\mu(G, S)\le 1$, and let $W$ be a subgraph of $G$.
Let $P$ be a path of $W$ whose all internal vertices have degree $2$ in $W$ such that 
$P_{mid}$ is non-empty, $W-V(P_{mid})$ contains an $S$-cycle, and $G$ has no $(W, P_{mid}, W-V(P_{mid}))$-path. Let $a$ and $b$ be the gates of $P$.

Then the following are satisfied.
\begin{enumerate}[(1)]
	\item $\{a,b\}$ separates $P_{mid}$ from $W-V(P_{mid})$ in $G$.
\item If $G$ has an $S$-cycle $C$ containing a vertex of $P_{mid}$,
then $C$ contains both $a, b$ and it also contains a vertex of $P_v$ for each endpoint $v$ of $P$. 
Furthermore, if $C$ does not contain an endpoint $v$ of $P$, then $C$ contains a $(W, P_v, W-V(P_v))$-path. 
\end{enumerate}
\end{lemma}
\begin{proof}
	(Proof of (1)) Suppose that $\{a, b\}$ does not separate $P_{mid}$ and $W-V(P_{mid})$ in $G$.
	Then a shortest path from $P_{mid}$ to $W-V(P_{mid})$ in $G-\{a,b\}$
	 is a $(W, P_{mid}, W-V(P_{mid}))$-path, a contradiction.
	
	\medskip

	(Proof of (2)) 
	By (1), $\{a, b\}$ separates $P_{mid}$ and $W-V(P_{mid})$ in $G$.
	Let $U$ be the connected component of $G-\{a, b\}$ containing $P_{mid}-\{a,b\}$.

	If $C$ is fully contained in $G[V(U)\cup \{a,b\}]$, 
	then it is vertex-disjoint from an  $S$-cycle in $W-V(P_{mid})$ given by the assumption.
	It means that $C$ is not fully contained in $G[V(U)\cup \{a,b\}]$ and it contains both $a$ and $b$.
	
	Now, let $v, w$ be the endpoints  of $P$ 
	such that $\dist_P(a, v)\le \dist_P(b, v)$.
	As $C$ is not fully contained in $G[V(U)\cup \{a,b\}]$, 
	$a$ has a neighbor in $C$ that is not contained in $V(U)\cup \{a,b\}$.
	Let $a'$ be such a neighbor.
	Assume that $a'$ is not the neighbor of $a$ in $P$.
	Then following the direction from $a$ to $a'$ in $C$, 
	either $C$ meets $W$ on exactly $a$ or 
	we can find a $(W, P_{mid}, W-V(P_{mid}))$-path.
	In both cases, they contradict with the given assumption.
	Therefore, $a'$ is the neighbor of $a$ in $P$, which further implies that 
	$C$ contains a vertex in $P_v$.
	 
	 Lastly, suppose that $v\notin V(C)$.
	By the symmetric argument, $C$ also contains a vertex in $P_w$.
	As $C-(V(U)\cup \{a,b\})$ is connected, 
	there should be a path from $P_v$ to $P_w$ in $G-(V(U)\cup \{a,b\})$. 
	As $v\notin V(C)$, 
	it implies that there is a $(W, P_v, W-V(P_v))$-path.
\end{proof}

	\begin{lemma}\label{lem:twocycles}
	Let $(G, S)$ be a graph such that $G$ contains an $S$-cycle subgraph $W$.
	Let $C$ be a cycle of $W$ and $v,w\in V(C)$ and $P$ be a $(W, G[\{v\}], G[\{w\}])$-path such that 
	\begin{itemize}
		\item for the two cycles $C_1$ and $C_2$ of $C\cup P$ other than $C$, 
		$W$ has two cycles $C_1'$ and $C_2'$ where $C_i$ is vertex-disjoint from $C_i'$ for each $i\in \{1,2\}$.
	\end{itemize}
	Then $G$ has two vertex-disjoint $S$-cycles.
	\end{lemma}
	\begin{proof}
	Since $W$ is an $S$-cycle subgraph, one of $C_1$ and $C_2$ is an $S$-cycle.
	If $C_i$ is an $S$-cycle, then $C_i$ and $C_i'$ are two vertex-disjoint $S$-cycles in $G$.
	\end{proof}
	
	\subsection{Finding an extension for $S$-cycle subgraphs}\label{subsec:extension}
	
	As explained in Section~\ref{sec:intro}, we will recursively find a larger $S$-cycle subgraph.
	The following lemma describes a way to find an extension.
	
\begin{lemma}\label{lem:intersection2}
	Let $(G, S)$ be a rooted graph with $\mu(G, S)\le 1$.
	Let $W$ be an $S$-cycle subgraph of $G$ containing a cycle, and $T\subseteq V(W)\cap S$ be an $S$-cycle hitting set of $W$. 
	If $G-T$ contains an $S$-cycle $C$, then either
	$\abs{V(C)\cap V(W)}= 1$ or
	$C$ contains a $W$-extension.
\end{lemma}
\begin{proof}
	Suppose that $G-T$ contains an $S$-cycle $C$.
	Because $\mu(G, S)\le 1$ and $W$ has an $S$-cycle, $C$ has to intersect $W$. 
	We may assume that $\abs{V(C)\cap V(W)}\ge 2$; otherwise, we have the first outcome.
	
	Suppose that $C$ contains a $W$-path $X$ containing a vertex of $S$.
	Then clearly, $W\cup X$ is an $S$-cycle subgraph, 
	because every cycle in $W\cup X$ going through $X$ contains  a vertex of $S$.
	Therefore, we may assume that $C$ contains no $W$-path $X$ containing a vertex of $S$, 
	which implies that it has to contain a vertex of $S\cap V(W)$.
	Let $v$ be a vertex of $S\cap V(W)$ contained in $C$.
	
	Observe that $W-T$ is a forest, because $W$ is an $S$-cycle subgraph.
	Let $F$ be a connected component of $W-T$ that contains $v$.
	As $F$ is a tree, either $v$ has degree $1$ in $F$, or $F-v$ is disconnected.
	
	Let $v_1$ and $v_2$ be the two neighbors of $v$ in $C$.
	Suppose one of them, say $v_1$,  is in $G-V(W)$. Then following the direction from $v$ to $v_1$ in $C$, 
	we can find a $W$-path $R$ whose one endpoint is $v$. 
	As $v\in S$, $W\cup R$ is an $S$-cycle subgraph.
	Thus, we may assume that $v_1, v_2\in V(W)$.
	Note that they have to be contained in distinct connected components of $F-v$.
	
	Because $v_1$ and $v_2$ are contained in distinct connected components of $W-(T\cup \{v\})$
	and $C-v$ is connected, there should be a $W$-path $Q$ whose endpoints are contained in distinct connected components of $W-(T\cup \{v\})$.
	Then $W\cup Q$ is an $S$-cycle subgraph. This is because every path connecting two endpoints of $Q$ in $W$
	has to meet at least one vertex of $T\cup \{v\}\subseteq S$.
	
	This concludes the lemma.
	\end{proof}

	In the proof of Theorem~\ref{thm:main1}, 
	we will assume $\tau(G, S)>4$ and obtain a contradiction.
	To apply Lemma~\ref{lem:intersection2} to find a larger $S$-cycle subgraph, 
	we need to find a set $T$ in the lemma that has size at most $4$.
	Lemma~\ref{lem:smallhitting} is useful to find such a small hitting set.
	
	\begin{lemma}\label{lem:edgehitting}
	Let $G$ be a connected graph and $F\subseteq E(G)$ such that 
	every cycle of $G$ contains an edge of $F$.
	Then $G$ contains an edge set $X\subseteq F$
	such that $\abs{X}\le \abs{E(G)}-\abs{V(G)}+1$ and $G-X$ has no cycles. 
	\end{lemma}
	\begin{proof}
	We prove by induction on $\abs{E(G)}$. 
	We may assume that $G$ has a cycle; otherwise, we may take $X=\emptyset$ as $0= \abs{E(G)}-\abs{V(G)}+1$.
	Let $C$ be a cycle of $G$, and let $e$ be an edge in $F\cap E(C)$.
	Note that $G-e$ is still connected, and every cycle of $G-e$ contains an edge of $F\setminus \{e\}$.
	By induction hypothesis, 
	$G-e$ contains an edge set $X'\subseteq F\setminus \{e\}$ such that
	\[\abs{X'}\le \abs{E(G-e)}-\abs{V(G-e)}+1=\abs{E(G)}-\abs{V(G)}\]
	and $G-e-X'$ has no cycles.
	Thus, $X=X'\cup \{e\}$ is a required set of edges.
	\end{proof}
	We can translate Lemma~\ref{lem:edgehitting} for
	$S$-cycle subgraphs.
		\begin{lemma}\label{lem:smallhitting}
	Let $(G, S)$ be an $S$-cycle $H$-subdivision for some connected graph $H$.
	Then $G$ contains a vertex set $U\subseteq S$
	such that $\abs{U}\le \abs{E(H)}-\abs{V(H)}+1$ and $G-U$ has no cycles. 
	\end{lemma}
	\begin{proof}
	Observe that a vertex $v$ of $S$ in $G$ hits all cycles along a certifying path containing $v$.
	Let $F$ be the set of edges of $H$ corresponding to the certifying paths of $G$ containing a vertex of $S$.
	
	By  Lemma~\ref{lem:edgehitting}, $H$ contains an edge set $X\subseteq F$ such that $\abs{X}\le \abs{E(H)}-\abs{V(H)}+1$ and $H-X$ has no cycles.
	By taking a vertex of $S$ for each certifying path corresponding to an edge of $X$, 
	we can find a vertex set $U\subseteq S$ such that 
	$\abs{U}\le \abs{E(H)}-\abs{V(H)}+1$ and $G-U$ has no cycles.
	\end{proof}

	The following lemma is another application of Lemma~\ref{lem:intersection2}.
	\begin{lemma}\label{lem:conclusion}
	Let $(G, S)$ be a rooted graph with $\mu(G, S)\le 1$ such that 
	$G$ contains an $S$-cycle $H$-subdivision $W$ for some graph $H$.	
	Let $T\subseteq V(W)$ such that $G-T$ has no $S$-cycle intersecting $W$ on at most $1$ vertex, and it has no $W$-extension.
	If $G-T$ has no $S$-cycle containing a vertex of $S\cap V(W)$, then $T$ is an $S$-cycle hitting set of $G$.
	\end{lemma}
	\begin{proof}
	Suppose that $G-T$ has an $S$-cycle $C$ and assume that 
	$C$ does not meet any vertex of $S\cap V(W)$.
	Then by Lemma~\ref{lem:intersection2}, it meets $W$ at one vertex, or it contains a $W$-extension avoiding $T$. 
	But this contradicts the assumption.
\end{proof}

	\subsection{Nice graphs}
	
		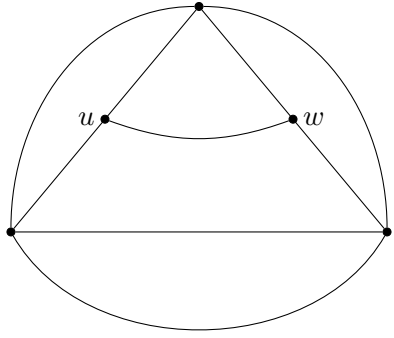
\begin{figure}
 \tikzstyle{v}=[circle, draw, solid, fill=black, inner sep=0pt, minimum width=3pt]
 \tikzstyle{w}=[rectangle, draw, solid, fill=black, inner sep=0pt, minimum width=5pt, minimum height=5pt]
  \centering
   \begin{tikzpicture}[scale=0.5]
      
        \node [v]  (v1) at (-5, 0){};
        \node [v]  (v2) at (0, 6){};
        \node [v]  (v3) at (5, 0){};
      
       \draw(v1)--(v2)--(v3)--(v1);
      \draw(v1) [in=180,out=90] to (v2);
    \draw(v3) [in=0,out=90] to (v2);
      \draw(v1) [in=-120,out=-60] to (v3);
      
     	\node [v] (a1) at (-2.5, 3) {};
     	\node [v] (a2) at (2.5, 3) {};
            \draw (a1) node [left] {$u$}; 
            \draw (a2) node [right] {$w$}; 
      \draw(a1) [in=-160,out=-20] to (a2);

\end{tikzpicture}
   \caption{$K_3^{++}$ is nice because it has no two vertex-disjoint cycles, but if we subdivide two edges and add an edge $uw$ as in figure, then we have two vertex-disjoint cycles.}
  \label{fig:nicegraph}
\end{figure}

	We say that a graph $H$ is \emph{nice} if $H$ has no two vertex-disjoint cycles, 
	but for any two distinct edges of $H$, 
	if we subdivide these edges once and add an edge between the new subdivided vertices, 
	then the obtained graph has two vertex-disjoint cycles.
	See Figure~\ref{fig:nicegraph} which depicts why $K_3^{++}$ is nice.
	Nice graphs have an additional property that for every edge $e_1$, there is a cycle that does not contain this edge.
	This is because if one choose another edge $e_2$ and subdivide both edges and add a new edge between two subdivided vertices, 
	then we have two vertex-disjoint cycles. Clearly, one of them does not contain the subdivided vertex from $e_1$, so originally, it does not contain  $e_1$.

	We mainly use the observation that  
	$K_3^{+++}$ and $K_4^{++}$ are nice.	This notion is useful in the following sense.
	
	\begin{lemma}\label{lem:nice}
	Let $H$ be a nice graph and let $(G, S)$ be a graph such that $G$ contains an $S$-cycle $H$-subdivision $W$.
	\begin{enumerate}
	\item 
	Let $P_1$ and $P_2$ be two distinct certifying paths of $W$ and 
	$X$ be a $(W, P_1', P_2')$-path that is a $W$-extension where for each $i\in \{1,2\}$, $P_i'$ is the path obtained from $P_i$ by removing its endpoints. Then $G$ has two vertex-disjoint $S$-cycles.
	\item If $G$ contains an $S$-cycle $C$ that intersects $W$ on exactly one vertex which is an internal vertex of a certifying path,
	then $G$ has two vertex-disjoint $S$-cycles.
	\end{enumerate}
	\end{lemma}
	\begin{proof}
	The first statement is clear from the definition of nice graphs, and the second statement is clear from the additional property of nice graphs.
	\end{proof}
	
	\subsection{Setting}
	In Sections~\ref{sec:k4subdiv} to \ref{sec:k33final}, we will prove 
	Theorem~\ref{thm:main1}, based on lemmas proved in this section.
	In these sections, we fix a rooted graph $(G, S)$ and
	suppose that $\mu(G, S)\le 1$ and $\tau(G, S)>4$. From this, we derive a contradiction at the end.
	In most of lemmas, an $S$-cycle $H$-subdivision $W$ for some $H$ will be given.
	For convenience, we will call a $(A, B)$-path for a $(W, A, B)$-path.

\section{Reduction to a $K_4$-subdivision}\label{sec:k4subdiv}

	In this section, we show that $G$ contains an $S$-cycle $K_4$-subdivision.  
	First we show that it contains an $S$-cycle $H$-subdivision for some
	$H\in \{K_3^{+++}, K_4\}$, and in the case when it contains an $S$-cycle $K_3^{+++}$-subdivision, 
	we prove that $\tau(G, S)\le 4$.
	
			\begin{lemma}\label{lem:reduction1}
	The graph $G$ contains an $S$-cycle $H$-subdivision for some 
	$H\in \{K_3^{+++}, K_4\}.$
\end{lemma}
\begin{proof}
	We first show that $G$ contains an $S$-cycle $H_1$-subdivision for some $H_1\in \{K_3^+, \theta_3\}$.
	As $\tau(G, S)>4$, $G$ contains an $S$-cycle.
	Let $C_1$ be an $S$-cycle of $G$, and $v\in S\cap V(C_1)$. 
	Again since $\tau(G, S)>4$, $G-v$ also contains an $S$-cycle, say $C_2$.
	Since $v\in S$, 
	by Lemma~\ref{lem:intersection2},
	either $C_1$ and $C_2$ meet at exactly one vertex, or
    $C_2$ contains a $C_1$-extension $X_1$.
    In the latter case, we have an $S$-cycle $\theta_3$-subdivision.
    Thus, we may assume that $C_1$ and $C_2$ meet at exactly one vertex.
    
    Note that $C_1\cup C_2$ can be seen as a subdivision of the graph on one vertex with two loops.
    Let $v_1$ be the intersection of $C_1$ and $C_2$.
    By Lemma~\ref{lem:smallhitting}, $C_1\cup C_2$ has an $S$-cycle hitting set $T_1\subseteq S\cap V(C_1\cup C_2)$ 
    such that $\abs{T_1}\le 2$.
    As $\tau(G, S)>4$, $G-(T_1\cup \{v\})$ contains an $S$-cycle, say $C_3$.
    By Lemma~\ref{lem:intersection2}, either $C_3$ and $C_1\cup C_2$ meet at exactly one vertex, or $C_3$ contains a
    $(C_1\cup C_2)$-extension $X_2$.
    In the former case, we have two vertex-disjoint $S$-cycles, a contradiction.
    In the latter case, the endpoints of $X_2$ have to be contained in distinct cycles, otherwise, we have two vertex-disjoint $S$-cycles.
    Thus, $G$ contains an $S$-cycle $K_3^+$-subdivision, as required.

\medskip	
	In the next, we show that 
	$G$ contains an $S$-cycle $H_2$-subdivision for some $H_2\in \{K_3^{++}, K_4\}$.
	We know that $G$ contains an $S$-cycle $H_1$-subdivision $W$ for some $H_1\in \{K_3^+, \theta_3\}$.
	
	Suppose that $W$ is an $S$-cycle $K_3^+$-subdivision. 
	Note that $\abs{E(K_3^+)}=5$ and $\abs{V(K_3^+)}=3$. 
	Thus, by Lemma~\ref{lem:smallhitting}, 
	$W$ contains an $S$-cycle hitting set $T_2\subseteq S\cap V(W)$ of size at most $3$.
	Let $w$ be the vertex incident with four edges in $W$, and let $v_1, v_2$ be the two other branching vertices.
	As $\tau(G, S)>4$, $G-(T_2\cup \{w\})$ contains an $S$-cycle, say $C_4$.
	
	By Lemma~\ref{lem:intersection2}, 
	either $C_4$ and $W$ meet at exactly one vertex, or 
	$C_4$ contains a $W$-extension $X_3$.
	In the former case, there are two vertex-disjoint $S$-cycles.
	So, the latter case holds.
	If the endpoints of $X_3$ are contained in the certifying path between $v_1$ and $v_2$, then 
	$W\cup X$ contains either two vertex-disjoint $S$-cycles, or an $S$-cycle  $K_3^{++}$-subdivision.
	So, we may assume that one endpoint of $X_3$ is in the certifying path between $w$ and $v_i$, as an internal vertex.
	If the other endpoint is contained in the same certifying path, then we have two vertex-disjoint $S$-cycles, 
	and otherwise, $G$ contains an $S$-cycle $K_4$-subdivision.
	
	Now, suppose that $W$ is an $S$-cycle $\theta_3$-subdivision.
	Note that $\abs{E(\theta_3)}=3$ and $\abs{V(\theta_3)}=2$. 
	Thus, by Lemma~\ref{lem:smallhitting}, 
	$W$ contains an $S$-cycle hitting set $T_3\subseteq S\cap V(W)$ of size at most $2$.
	Let $w_1, w_2$ be the branching vertices of $W$.
	As $\tau(G, S)>4$, $G-(T_3\cup \{w_1, w_2\})$ contains an $S$-cycle, say $C_5$.
	
	By Lemma~\ref{lem:intersection2}, 
	either $C_5$ and $W$ meet at exactly one vertex, or 
	$C_5$ contains a $W$-extension $X_4$.
	In the former case, there are two vertex-disjoint $S$-cycles.
	So, the latter case holds.
	If the endpoints of $X_4$ are contained in the same certifying path, then 
	$W\cup X_4$ contains two vertex-disjoint $S$-cycles.
	Thus, the two endpoints of $X_4$ are contained in distinct certifying paths, 
	and $G$ contains an $S$-cycle $K_4$-subdivision.
	
	\medskip
	Lastly, we show that 
	if $G$ contains an $S$-cycle $K_3^{++}$-subdivision, 
	then it contains an $S$-cycle $K_3^{+++}$-subdivision or an $S$-cycle $K_4$-subdivision.
	Suppose that 
	$G$ contains an $S$-cycle $K_3^{++}$-subdivision $W'$.
	Note that $\abs{E(K_3^{++})}=6$ and $\abs{V(K_3^{++})}=3$. 
	Thus, by Lemma~\ref{lem:smallhitting}, 
	$W'$ contains an $S$-cycle hitting set $T_3\subseteq S\cap V(W')$ of size at most $4$.
	
	As $\tau(G, S)>4$, $G-T_3$ contains an $S$-cycle, say $C_6$.
	By Lemma~\ref{lem:intersection2}, 
	either $C_6$ and $W'$ meet at exactly one vertex, or 
	$C_6$ contains a $W'$-extension.
	In the former case, we have two vertex-disjoint $S$-cycles.
	So, $C_6$ contains a $W'$-extension, say $X_5$.
	If both endpoints of $X_5$ are branching vertices, 
	then $G$ contains an $S$-cycle $K_3^{+++}$-subdivision, and we are done.
	Thus, we may assume that one of the endpoints of $X_5$ is an internal vertex of a certifying path of $W'$.
	The other endpoint of $X_5$ is contained in the same certifying path, then we have two vertex-disjoint $S$-cycles.
	Otherwise, $G$ contains an $S$-cycle $K_4$-subdivision, as required.
\end{proof}

\begin{proposition}\label{prop:k3triple}
	If $G$ contains an $S$-cycle $K_3^{+++}$-subdivision, then $\tau (G, S)\le 3$.
\end{proposition}
\begin{proof}
	Let $W$ be an $S$-cycle $K_3^{+++}$-subdivision of $G$.
	Let $v_1, v_2, v_3$ be the branching vertices of $W$, 
	and let $P^1, P^2, P^3$ be the certifying paths from $v_1$ to $v_2$, 
	and $Q^1, Q^2$ be the certifying paths from $v_2$ to $v_3$, and 
	$R^1, R^2$ be the certifying paths from $v_3$ to $v_1$.
	Let $T=\{v_1, v_2, v_3\}$.
	Since $K_3^{+++}$ is nice and $\mu(G, S)\le 1$, $G-T$ has no $W$-extension and has no $S$-cycle meeting $W$ on one vertex.
	
	We claim that $G-T$ has no $S$-cycle containing a vertex in $S\cap V(W)$.
	If this is true, then by Lemma~\ref{lem:conclusion}, 
	we conclude that $T$ is an $S$-cycle hitting set and thus $\tau (G, S)\le 3$.
	
	By applying Lemma~\ref{lem:twocycles}, 
	we can observe that in $G-T$, 
	\begin{itemize}
	\item there is no $( P^i, P^j)$-path for distinct $i, j\in \{1, 2, 3\}$, 
	\item there is no $(Q^1, Q^2)$-path, and 
	\item there is no $(R^1, R^2)$-path.
	\end{itemize}

	\begin{claim}\label{claim:k3+++pipj}
 	Let $i\in \{1,2,3\}$. No $S$-cycle in $G-T$ contains a vertex of $P^i_{mid}$.
	\end{claim}
	\begin{clproof}
    It suffices to prove for $i=1$. 
    Suppose that an $S$-cycle $H$ in $G-T$ contains a vertex of $P^1_{mid}$.
	As $G-T$ has no $W$-extension, 
	there is no $(P^1_{mid}, W-V(P^1_{mid}))$-path.	
	Also, observe that $W-V(P^1_{mid})$ contains an $S$-cycle.

	Because $v_1, v_2\in T$, by Lemma~\ref{lem:mid}, $H$ contains a $(P^1_{v_j}, W-V(P^1_{v_j}))$-path $X_j$ for each $j\in \{1, 2\}$.
	The endpoint of $X_j$ in $W-V(P^1_{v_j})$ is not contained in $P^1\cup P^2\cup P^3$.
	Therefore, $P^1\cup X_1\cup X_2\cup Q^1\cup Q^2\cup R^1\cup R^2-\{v_1, v_2\}$ contains  
	an $S$-cycle avoiding $P^2\cup P^3$.  
	This is a contradiction.
	\end{clproof}
	
	Next we show that no $S$-cycle in $G-T$ contains a vertex of $Q^1_{mid}$. If this is true, then by symmetry, 
	$G-T$ has no $S$-cycle containing a vertex in $S\cap V(W)$.

	Suppose for contradiction that $G-T$ contains an $S$-cycle $H$ containing a vertex of $Q^1_{mid}$.
	As in Claim~\ref{claim:k3+++pipj}, 
	by Lemma~\ref{lem:mid}, $H$ contains a $(Q^1_{v_2}, W-V(Q^1_{v_2}))$-path $Y$.
	Then the endpoint of $Y$ in $W-V(Q^1_{v_2})$ should be in $P^1\cup P^2\cup P^3$; otherwise, 
	we can find an $S$-cycle vertex-disjoint from $P^1\cup P^2$.

	\begin{claim}\label{claim:k3+++compo}
	Let $i\in \{1,2,3\}$. If there is a $(P^i, Q^1_{v_2})$-path in $G-T$, then any $(P^i, W-V(P^i))$-path in $G-T$ satisfies that 
	its endpoint in $W-V(P^i)$ is contained in $Q^1_{v_2}$.
	\end{claim}
	\begin{clproof}
	Suppose $i=1$ and  
	there is a $(P^1, W-V(P^1)-V(Q^1_{v_2}))$-path $Z$ in $G-T$.
	Then the endpoint of $Z$ in $W-V(P^1)-V(Q^1_{v_2})$ is contained in $Q^1\cup Q^2\cup R^1\cup R^2$.
	It implies that there is an $S$-cycle avoiding $P^2\cup P^3$, a contradiction.
	So, $G-T$ has no $(P^1, W-V(P^1)-V(Q^1_{v_2}))$-path.
	\end{clproof}
	
	One can also observe that if there is a $(Q^1_{v_2}, X)$-path in $G-T$ where $X=(Q^1\cup Q^2\cup R^1\cup R^2)-V(Q^1_{v_2})$, then $G$ contains two vertex-disjoint $S$-cycles. So, such a path does not exist.
	
	Let $z\in V(Q^1_{mid})$ such that $\dist_{Q^1}(z, v_2)$ is minimum.
	Let $I\subseteq \{1, 2, 3\}$ be the set such that for every $i\in I$, there is a $(P^i, Q^1_{v_2})$-path in $G-T$.
	By Claim~\ref{claim:k3+++compo},
	$z$ separates $(\bigcup_{i\in I}P^i) \cup Q^1_{v_2}$ and the rest of $W$ in $G-T$.
	But the $S$-cycle $H$ contains a vertex of $Q^1_{v_2}$ and a vertex of $Q^1_{v_3}$ by Lemma~\ref{lem:mid}.
	This is a contradiction.
\end{proof}

\section{Reduction to a $W_4$ or $K_{3,3}^+$-subdivision}\label{sec:w4ork33}

By Lemma~\ref{lem:reduction1} and Proposition~\ref{prop:k3triple}, 
we know that $G$ contains an $S$-cycle $K_4$-subdivision.
In this section, we prove that $G$ contains an $S$-cycle $H$-subdivision for some 
			$H\in \{K_4^{++}, K_4^{+++}, W_4, K_{3,3}^+\}$, 
			and in the case when $G$ contains an $S$-cycle $H$ subdivision for some 
			$H\in \{K_4^{++}, K_4^{+++}\}$, 
			$\tau (G, S)\le 4$.

\begin{lemma}\label{lem:reduction2}
	If $G$ has an $S$-cycle $K_4$-subdivision,
	then it contains an $S$-cycle $H$-subdivision for some 
			$H\in \{K_4^{++}, K_4^{+++}, W_4, K_{3,3}^+\}$.
\end{lemma}
\begin{proof}

	Suppose that $G$ has an $S$-cycle $K_4$-subdivision $W$.
	
\begin{claim}\label{claim:reduction3}
The graph $G$ contains an $S$-cycle $H_1$-subdivision for some $H_1\in \{K_4^+, W_4, K_{3,3}\}$.
\end{claim}
\begin{clproof}
	Note that $\abs{E(K_4)}=6$ and $\abs{V(K_4)}=4$. 
	Thus, by Lemma~\ref{lem:smallhitting}, 
	$W$ contains an $S$-cycle hitting set $T\subseteq S\cap V(W)$ such that $\abs{T}\le 3$.
	As $\tau(G, S)>4$, $G-T$ contains an $S$-cycle, say $C_1$.
	By Lemma~\ref{lem:intersection2}, 
	either $C_1$ and $W$ meet at exactly one vertex, 
	or $C_1$ contains a $W$-extension $X_1$.
	In the former case, we have two vertex-disjoint $S$-cycles.
	So, we may assume that the latter statement holds.
	
	If both endpoints of $X_1$ are branching vertices of $W$, then 
	$W\cup X_1$ is an $S$-cycle $K_4^+$-subdivision.
	Assume that exactly one endpoint of $X_1$ is a branching vertex.
	If the endpoints of $X_1$ are contained in a same certifying path, then $W\cup X_1$ contains two vertex-disjoint $S$-cycles.
	Otherwise, $G$ contains an $S$-cycle $W_4$-subdivision.
	
	Lastly, suppose that both endpoints are not branching vertices.
	If the certifying paths containing these endpoints share an endpoint, then 
	there are two vertex-disjoint $S$-cycles. Otherwise, $G$ contains an $S$-cycle $K_{3,3}$-subdivision, as required.
	\end{clproof}
	
	We repeat a similar argument to find a $K_4^{++}$-subdivision or a $K_4^{+++}$-subdivision.
	\begin{claim}\label{claim:reductionk4++}
	If $G$ contains an $S$-cycle $K_4^+$-subdivision, then it contains an $S$-cycle $H_2$-subdivision for some $H_2\in \{K_4^{++}, K_4^{+++}, W_4, K_{3,3}\}$.
	\end{claim}
	\begin{clproof}
	Let $W'$ be an $S$-cycle $K_4^{+}$-subdivision.
	Note that $\abs{E(K_4^+)}=7$ and $\abs{V(K_4^+)}=4$. 
	Thus, by Lemma~\ref{lem:smallhitting}, 
	$W'$ contains an $S$-cycle hitting set $T\subseteq S\cap V(W')$ such that $\abs{T}\le 4$.
	As $\tau(G, S)>4$, $G-T$ contains an $S$-cycle, say $C_2$.
	By Lemma~\ref{lem:intersection2}, 
	either $C_2$ and $W'$ meet at exactly one vertex, 
	or $C_2$ contains a $W'$-extension $X_2$.
	In the former case, we have two vertex-disjoint $S$-cycles.
	So, we may assume that the latter statement holds.

	If the both endpoints of $X_2$ are branching vertices of $W'$, then 
	it contains an $S$-cycle $K_4^{++}$-subdivision or an $S$-cycle $K_4^{+++}$-subdivision or two vertex-disjoint $S$-cycles.
	When $X_2$ has at most one branching vertex as an endpoint, by the same argument in Claim~\ref{claim:reduction3}, 
	we can find an $S$-cycle $W_4$-subdivision or an $S$-cycle $K_{3,3}$-subdivision.
	\end{clproof}
	
	We show that 
	if an $S$-cycle $K_{3,3}$-subdivision exists, 
	then there is an $S$-cycle $K_{3,3}^+$-subdivision.
	
	\begin{claim}\label{claim:reduction5}
If $G$ contains an $S$-cycle $K_{3,3}$-subdivision, then it contains an $S$-cycle $K_{3,3}^+$-subdivision.
\end{claim}
\begin{clproof}
Let $W''$ be an $S$-cycle $K_{3,3}$-subdivision.
Note that $\abs{E(K_{3,3})}=9$ and $\abs{V(K_{3,3})}=6$. 
	Thus, by Lemma~\ref{lem:smallhitting}, 
	$W''$ contains an $S$-cycle hitting set $T\subseteq S\cap V(W'')$ of size at most $4$.
	As $\tau(G, S)>4$, $G-T$ contains an $S$-cycle, say $C_3$.
	By Lemma~\ref{lem:intersection2}, 
	either $C_3$ and $W''$ meet at exactly one vertex, 
	or $C_3$ contains a $W''$-extension $X_3$.
	In the former case, we have two vertex-disjoint $S$-cycles.
	So, we may assume that the latter statement holds.
	
	Assume that the both endpoints of $X_3$ are branching vertices of $W''$.
	If both endpoints are contained in the same part of the bipartition of $K_{3,3}$, then $G$ contains an $S$-cycle $K_{3,3}^+$-subdivision.
	Otherwise, $G$ has two vertex-disjoint $S$-cycles, a contradiction.
	
	Now, assume that at most one endpoint of $X_3$ is a branching vertex of $W''$.
	Let $v, w$ be the endpoints  of $X_3$.
	Since one of $v$ and $w$ is not a branching vertex, 
	$W''$ has a path from $v$ to $w$ in $W''$ containing at most one branching vertex from each of the bipartition.
	Thus, other 4 branching vertices with certifying paths between them provide an $S$-cycle disjoint from one created by $X_3$ and the path from $v$ to $w$.
	So, $G$ contains two vertex-disjoint $S$-cycles, a contradiction.
\end{clproof}
We conclude that $G$ contains an $S$-cycle $H$-subdivision for some 
			$H\in \{K_4^{++}, K_4^{+++}, W_4, K_{3,3}^+\}$.
\end{proof}

Next, we focus on the case when
$G$ contains an $S$-cycle $K_4^{++}$-subdivision.

	\begin{figure}
 \tikzstyle{v}=[circle, draw, solid, fill=black, inner sep=0pt, minimum width=3pt]
 \tikzstyle{w}=[rectangle, draw, solid, fill=black, inner sep=0pt, minimum width=5pt, minimum height=5pt]
  \centering
   \begin{tikzpicture}[scale=2]

		\node [v, label=below:$v_1$]  (c1) at (0, 0-5){};
        \node [v, label=below:$v_2$]  (c3) at (2, 0-5){};
        \node [v, label=above:$v_4$]  (c2) at (1, 1.7-5){};
        \node [v, label=below:$v_3$]  (c4) at (1, 0.8-5){};
	
		\draw(c1) to [edge node={node [left] {$R^1_1$}}] (c2);
		\draw(c2) to [edge node={node [right] {$R^2_1$}}] (c3);
		\draw(c3) to [edge node={node [below] {$Q^1$}}] (c1);

		\draw(c4) to [edge node={node [right] {$Q^3$}}] (c1);        
		\draw(c4) to [edge node={node [right] {$R^3$}}] (c2);        
		\draw(c4) to [edge node={node [left] {$Q^2$}}] (c3);        
          \draw(c1) [in=180,out=120] to [edge node={node [left] {$R^1_2$}}] (c2);
          \draw(c3) [in=0,out=60] to [edge node={node [right] {$R^2_2$}}] (c2);

\end{tikzpicture}
   \caption{The $K_4^{++}$-subdivision in Proposition~\ref{prop:k4++}.}
  \label{fig:K4++}
\end{figure}
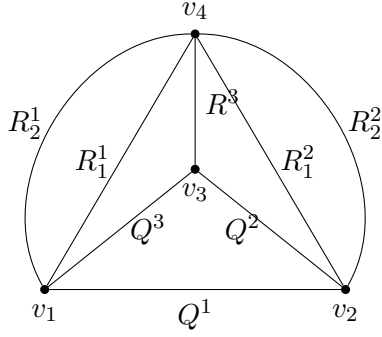

	\begin{proposition}\label{prop:k4++}
If $G$ contains an $S$-cycle $K_4^{++}$-subdivision, then $\tau (G, S)\le 4$.	
	\end{proposition}
	\begin{proof}
	Let $W$ be an $S$-cycle $K_4^{++}$-subdivision of $G$.
	Let $v_1, v_2, v_3, v_4$ be the branching vertices of $W$, 
	and for each $i\in \{1, 2\}$, let $Q^i$ be the certifying path from $v_i$ to $v_{i+1}$, 
	and $Q^3$ be the certifying path from $v_3$ to $v_1$, and 
	for each $j\in \{1, 2\}$, $R^j_1, R^j_2$ be the two certifying paths from $v_4$ to $v_j$,
	and $R^3$ be the certifying path from $v_4$ to $v_3$. 
	See Figure~\ref{fig:K4++} for an illustration.
	Let $T$ be the set obtained from $\{v_1, v_2, v_4\}$ by 
	\begin{itemize}
	\item adding a gate of $Q^1$ if $V(Q^1_{mid})$ is non-empty, 
	\item adding a vertex $w\in S$ on $Q^2\cup Q^3$ where $\dist_{Q^2\cup Q^3}(w, v_3)$ is minimum, otherwise.
	\end{itemize}
	Clearly, $\abs{T}\le 4$. 
	
	Recall that $K_4^{++}$ is nice. 
	Observe that there is no $W$-extension in $G-T$ whose one endpoint is $v_3$; if there is such an extension, 
	then it creates an $S$-cycle disjoint from one of $R^1_1\cup R^1_2$ and $R^2_1\cup R^2_2$.
	Also, there is no $S$-cycle meeting $W$ on exactly $v_3$.
	Thus, $G-T$ has no $W$-extension and has no $S$-cycle meeting $W$ on one vertex.
		
	We will show that $G-T$ has no $S$-cycle containing a vertex in $S\cap V(W)$.
	If this is true, then by Lemma~\ref{lem:conclusion}, $T$ is an $S$-cycle hitting set and thus $\tau (G, S)\le 4$.

	By applying Lemma~\ref{lem:twocycles} appropriately, 
	we can observe that in $G-T$, 
	\begin{itemize}
	\item there is no $(R^i_1, R^i_2)$-path for each $i\in \{1, 2\}$, 
	\item there is no $(R^i_j, R^3-v_3)$-path for $i, j\in \{1, 2\}$, and
	\item there is no $(Q^1, Q^2\cup Q^3)$-path.
	\end{itemize}	
	
	\begin{claim}\label{claim:k4++rijmid}
 Let $i, j\in \{1, 2\}$. No $S$-cycle in $G-T$ contains a vertex of $(R^i_j)_{mid}$.
	\end{claim}
	\begin{clproof}
	It is sufficient to show for $i=j=1$ by symmetry. Suppose that an $S$-cycle $H$ in $G-T$ contains a vertex of $(R^1_1)_{mid}$.
	As $G-T$ has no $W$-extension, 
	by Lemma~\ref{lem:mid}, $H$ contains a $((R^1_1)_{v_k}, W-V((R^1_1)_{v_k}))$-path $X_k$ for each $k\in \{1, 4\}$.

	If the endpoint of $X_1$ in $W-V((R^1_1)_{v_1})$ is not contained in 
	$Q^1\cup Q^2\cup Q^3$, then one can find an $S$-cycle disjoint from $Q^1\cup Q^2\cup Q^3$.
	Also, 
	if the endpoint of $X_4$ in $W-V((R^1_1)_{v_4})$ is not contained in $R^2_1\cup R^2_2$, 
	then one can find an $S$-cycle disjoint from $R^2_1\cup R^2_2$.
	So, we may assume that an endpoint of $X_1$ is contained in $Q^1\cup Q^2\cup Q^3$, and an endpoint of $X_4$ is contained in $R^2_1\cup R^2_2$.
	Without loss of generality, we may assume that the endpoint of $X_4$ is contained in $R^2_1$.
	
	Observe that 
	if there is an $(R^1_1, R^2_2)$-path, then there is an $S$-cycle disjoint from one of $Q^1\cup Q^2\cup Q^3$ and $R^1_2\cup R^3\cup Q^3$. Thus, there is no $(R^1_1, R^2_2)$-path, and  
	similarly, there is no $(R^2_1, R^1_2)$-path. 
	
	Because of $X_4$, the endpoint of $X_1$ in $Q^1\cup Q^2\cup Q^3$ is contained in $Q^3$; otherwise, we can find an $S$-cycle disjoint from $R^1_2\cup R^3\cup Q^3$.
	Also, in $G-T$, there is no $(R^2_1, Q^1\cup Q^3-v_3)$-path because of $R^2_2\cup R^3\cup Q^2$.
	Also, we already observed that there is no $(R^2_1, R^3-v_3)$-path.

	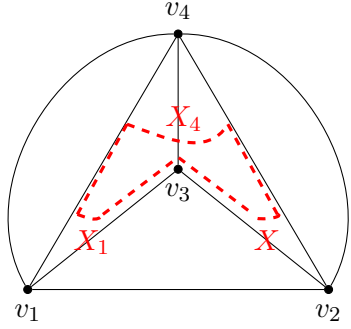
\begin{figure}
 \tikzstyle{v}=[circle, draw, solid, fill=black, inner sep=0pt, minimum width=3pt]
 \tikzstyle{w}=[rectangle, draw, solid, fill=black, inner sep=0pt, minimum width=5pt, minimum height=5pt]
  \centering
   \begin{tikzpicture}[scale=2]

		\node [v, label=below:$v_1$]  (c1) at (0, 0-5){};
        \node [v, label=below:$v_2$]  (c3) at (2, 0-5){};
        \node [v, label=above:$v_4$]  (c2) at (1, 1.7-5){};
        \node [v, label=below:$v_3$]  (c4) at (1, 0.8-5){};
	
		\draw(c1) to (c2);
		\draw(c2) to (c3);
		\draw(c3) to (c1);

		\draw(c4) to (c1);        
		\draw(c4) to (c2);        
		\draw(c4) to (c3);        
          \draw(c1) [in=180,out=120] to (c2);
          \draw(c3) [in=0,out=60] to (c2);

	\draw[very thick, red, dashed] (0.66, 1.1-5) [in=-120, out=-20] to  [edge node={node [above] {$X_4$}}]  (1.33, 1.1-5);
	\draw[very thick, red, dashed] (0.33, 0.5-5) [in=-120, out=-20] to [edge node={node [below] {$X_1$}}] (0.5, 0.5-5);
	\draw[very thick, red, dashed] (1.66, 0.5-5) [in=-120, out=-20] to [edge node={node [below] {$X$}}] (1.5, 0.5-5);
	\draw[very thick, red, dashed] (0.66, 1.1-5)--(0.33, 0.5-5);
	\draw[very thick, red, dashed] (0.5, 0.5-5)--(1, 0.88-5)--(1.5, 0.5-5);
	\draw[very thick, red, dashed] (1.66, 0.5-5)--(1.33, 1.1-5);

\end{tikzpicture}
   \caption{The three paths $X_1, X_4, X$ in Claim~\ref{claim:k4++rijmid} of Proposition~\ref{prop:k4++}. The dotted $S$-cycle along $X_1, X_4, X$ is disjoint from $Q^1\cup R^2_2\cup R^1_2$.}
  \label{fig:K4++claim}
\end{figure}
	
	Thus, we may assume that there is an $(R^2_1, Q^2)$-path; otherwise a gate $z$ of $R^1_1$ separates $(R^1_1)_{v_1}$ and $(R^1_1)_{v_4}$ in $G-T$, contradicting that $H-z$ is a path meeting both parts. Call this path $X$.
	See Figure~\ref{fig:K4++claim} for an illustration.
	Then the $S$-cycle in $R^1_1\cup R^2_1\cup Q^2\cup Q^3\cup X_1\cup X_4\cup X$ that does not meet $v_1, v_2, v_4$ is disjoint from $Q^1\cup R^2_2\cup R^1_2$, a contradiction.		
	Therefore no $S$-cycle in $G-T$ contains a vertex of $(R^1_1)_{mid}$.
	\end{clproof}

	\begin{claim}\label{claim:k4++r3mid}
	No $S$-cycle in $G-T$ contains a vertex of $R^3_{mid}$.
	\end{claim}
	\begin{clproof}
	Suppose that an $S$-cycle $H$ in $G-T$ contains a vertex of $R^3_{mid}$.
	As $G-T$ has no $W$-extension,
	by Lemma~\ref{lem:mid}, $H$ contains an $(R^3_{v_4}, W-V(R^3_{v_4}))$-path, say $X$.
	We observed that the endpoint of $X$ in $W-V(R^3_{v_4})$ is not contained in $R^2_1\cup R^2_2$.
	But then $W\cup X$ has an $S$-cycle vertex-disjoint from $R^2_1\cup R^2_2$, a contradiction.
	\end{clproof}

	By the construction of $T$, if $Q^1_{mid}$ is non-empty, then we added a gate of $Q^1$.
	It means that by Lemma~\ref{lem:mid}, no $S$-cycle in $G-T$ contains a vertex of $Q^1_{mid}$.
	Thus, it remains to show that no $S$-cycle in $G-T$ contains a vertex of $S$ on $Q^2\cup Q^3$.
	
	\begin{claim}\label{claim:k4++q2mid}
	No $S$-cycle in $G-T$ contains a vertex of $S$ on $Q^2\cup Q^3$.
	\end{claim}
	\begin{clproof}
	Suppose that $G-T$ has an $S$-cycle $H$ containing a vertex $u\in S$ on $Q^2\cup Q^3$.
	By symmetry, we may assume that $u\in V(Q_2)$.
	First claim that the two neighbors of $u$ in $H$ are neighbors of $u$ in $W$.
	Suppose for contradiction that there is a neighbor $u'$ of $u$ in $H$ that is not a neighbor in $W$.
    As $H$ is a cycle, following the direction from $u$ to $u'$, either 
   one can find a $W$-extension in $G-T$, or
    $H$ meets exactly $u$ on $W$.
   We know that both cases are not possible. So, we conclude that the two neighbors of $u$ in $H$ are neighbors of $u$ in $W$.

   We additionally claim that if $u=v_3$, then a neighbor of $u$ in $H$ is not in $R^3-\{v_3, v_4\}$. 
   Suppose that a neighbor of $u$ in $H$ is contained in $R^3-\{v_3, v_4\}$.
    As $H-u$ is a path, there is an $(R^3, W-V(R^3))$-path in $G-\{v_1, v_2, v_4, u\}$. 
    Then it creates an $S$-cycle disjoint from one of $R^1_1\cup R^1_2$ and $R^2_1\cup R^2_2$.
    So, the claim holds, and furthermore, when $u=v_3$, there is no $(R^3, W-V(R^3))$-path.
    It means that the two neighbors of $u$ in $H$ are contained in $Q^2\cup Q^3$.

	Let $u_1$ and $u_2$ be the two neighbors of $u$ in $H$ such that 
	\[\dist_{Q^2\cup Q^3}(v_2, u_1)< \dist_{Q^2\cup Q^3}(v_2, u_2).\]
	For each $i\in \{1, 2\}$, let $A_i$ be the connected component of $W-\{v_1, v_2, v_4, u\}$ containing $u_i$.
   Note that when $u=v_3$, we have that $A_1=Q^2-\{v_2, v_3\}$ and $A_2=Q^3-\{v_1, v_3\}$.

    As $H-u$ is a path, for each $i\in \{1, 2\}$, there is a $(A_i, W-V(A_i))$-path, say $X_i$.
    As $G$ has no two vertex-disjoint $S$-cycles, we may assume that 
     the endpoint of $X_1$ in $W-V(A_1)$ is contained in $R^2_1\cup R^2_2$, 
     and the endpoint of $X_2$ in $W-V(A_2)$ is contained in $R^1_1\cup R^1_2$.
     Recall that $w$ is the vertex in $T\setminus \{v_1, v_2, v_4\}$.
     
     We divide into two cases depending on whether $Q^1_{mid}$ is empty or not.
     
     \begin{itemize}
     	\item (Case 1. $Q^1_{mid}$ is empty.)
	
	In this case, $w$ is contained in $Q^2\cup Q^3$. Since $u\neq w$, we have $u\neq v_3$ and furthermore, $A_2$ contains $w$ because of the property that $w$ is chosen as a vertex of $S$ closest to $v_3$. 
	Let $A_2'$ be the component of $A_2-w$ that contains $u_2$.
	As $w\in T$ and $H-u$ is a path, there is a $(A_2', W-V(A_2'))$-path in $G-(T\cup \{u\})$.
	However, it is not difficult to check that there is no $(A_2', W-V(A_2'))$-path in $G-(T\cup \{u\})$, because $w$ and $u$ are in $S$.
	So, $H$ cannot exist.

	\item (Case 2. $Q^1_{mid}$ is non-empty.)

	By the construction of $T$, $w$ is a gate of $Q^1$, and thus $H$ contains no vertex of $Q^1_{mid}$.
	We first deal with the case when $u\neq v_3$.
	We introduce an auxiliary graph $F$ on the vertex set 
	\[ \{R^1_2, R^2_2, R^2_1, R^2_2, A_1, A_2, Q^1_{v_1}, Q^1_{v_2}\} \] 
	such that for $A, B\in V(F)$, $A$ is adjacent to $B$ if and only if there is an $(A, B)$-path in $G-T$.
	
	It is not difficult to see that $N_F(A_2)\subseteq \{R^1_1, R^1_2\}$ and $N_F(Q^1_{v_1})\subseteq \{R^1_1, R^1_2\}$, and symmetrically, 
	$N_F(A_1)\subseteq \{R^2_1, R^2_2\}$ and $N_F(Q^1_{v_2})\subseteq \{R^2_1, R^2_2\}$.
	We observe that if $Q^1_{v_1}$ is adjacent to $R^1_i$ in $F$ for some $i\in \{1,2\}$, 
	then $R^1_i$ has no neighbor in $\{R^2_1, R^2_2\}$; if there is such a neighbor, then we can find an $S$-cycle disjoint from $R^1_{3-i}\cup Q^3\cup R^3$.
	Symmetrically, if $Q^1_{v_2}$ is adjacent to $R^2_i$ in $F$ for some $i\in \{1,2\}$, 
	then $R^2_i$ has no neighbor in $\{R^1_1, R^1_2\}$.
	
	Now, we show that there is no path from $A_1$ to $A_2$ in $F$.
	Suppose there is a path $M$ from $A_1$ to $A_2$ in $F$. 
	By the above observation, we can see that $M$ contains an edge between $R^1_i$ and $R^2_j$ for some $i, j\in \{1,2\}$.
	Then $R^1_i$ is not adjacent to $R^1_{3-i}$ and $Q^1_{v_1}$, so it has to be adjacent to $A_2$, 
	and similarly, $R^2_j$ is adjacent to $A_1$.
	Then we can find an $S$-cycle using $W$-paths corresponding to $A_2-R^1_i-R^2_j-A_1$, 
	which is disjoint from $Q^1\cup R^1_{3-i}\cup R^2_{3-j}$.
	This contradicts to that $H-u$ contains a vertex of $A_1$ and a vertex of $A_2$.
	
	When $u=v_3$, we observed that there is no $(R^3, W-V(R^3))$-path. Thus, by the same argument as above, 
	we can derive a contradiction.

	\end{itemize}
    We conclude that 
    no $S$-cycle in $G-T$ contains a vertex of $S$ on $Q^2\cup Q^3$.     \end{clproof}
    By Claims~\ref{claim:k4++rijmid}, \ref{claim:k4++r3mid}, and \ref{claim:k4++q2mid}, 
    no $S$-cycle in $G-T$ contains a vertex of $S$ in $W$, as required.
	\end{proof}

	\begin{figure}
 \tikzstyle{v}=[circle, draw, solid, fill=black, inner sep=0pt, minimum width=3pt]
 \tikzstyle{w}=[rectangle, draw, solid, fill=black, inner sep=0pt, minimum width=5pt, minimum height=5pt]
  \centering
   \begin{tikzpicture}[scale=2]

		\node [v, label=below:$v_1$]  (c1) at (0, 0-5){};
        \node [v, label=below:$v_3$]  (c3) at (2, 0-5){};
        \node [v, label=above:$v_2$]  (c2) at (1, 1.7-5){};
        \node [v, label=below:$v_4$]  (c4) at (1, 0.8-5){};
	
		\draw(c1) to [edge node={node [right] {$P^1$}}] (c2);
		\draw(c2) to [edge node={node [right] {$R^2$}}] (c3);
		\draw(c3) to [edge node={node [below] {$R^1$}}] (c1);

		\draw(c4) to [edge node={node [right] {$Q^1$}}] (c1);        
		\draw(c4) to [edge node={node [right] {$Q^2$}}] (c2);        
		\draw(c4) to [edge node={node [left] {$Q^3$}}] (c3);        
          \draw(c1) [in=180,out=120] to [edge node={node [left] {$P^3$}}] (c2);
          \draw(c1) [in=200,out=90] to [edge node={node [right] {$P^2$}}] (c2);

\end{tikzpicture}
   \caption{The $K_4^{+++}$-subdivision in Proposition~\ref{prop:k4+++}.}
  \label{fig:K4+++}
\end{figure}
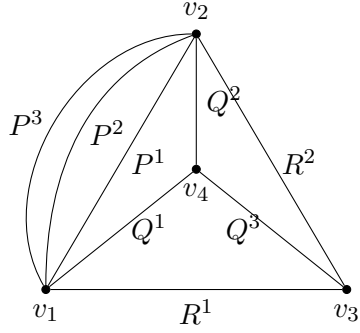

	\begin{proposition}\label{prop:k4+++}
If $G$ contains an $S$-cycle $K_4^{+++}$-subdivision, then 
either $G$ contains an $S$-cycle $K_{3,3}^+$-subdivision or $\tau (G, S)\le 4$.	
	\end{proposition}
		\begin{proof}
	Let $W$ be an $S$-cycle $K_4^{+++}$-subdivision in $G$.
	Let $v_1, v_2, v_3, v_4$ be the branching vertices of $W$, 
	let $P^1, P^2, P^3$ be the certifying paths from $v_1$ to $v_2$, 
	let $Q^j$ be the certifying path from $v_4$ to $v_j$ for $j\in \{1,2,3\}$, and 
	let $R^k$ be the certifying path from $v_3$ to $v_k$ for $k\in \{1,2\}$. 
	See Figure~\ref{fig:K4+++} for an illustration. Let $T=\{v_1,v_2,v_3,v_4\}$. 	
	
	If there is a $W$-extension $X$ in $G-T$ whose one endpoint is in $P^i$ and the other endpoint is in $Q^3$, 
	then $W\cup X$ has an $S$-cycle $K_{3,3}^+$-subdivision. Thus, we may assume that such an extension does not exist.
	It implies that $G-T$ has no $W$-extensions. Also, $G-T$ has no $S$-cycle hitting $W$ on one vertex.

	We will show that $G-T$ has no $S$-cycle containing a vertex in $S\cap V(W)$.
	If this is true, then by Lemma~\ref{lem:conclusion}, $T$ is an $S$-cycle hitting set and thus $\tau (G, S)\le 4$.
		By applying Lemma~\ref{lem:twocycles}, 
	we can observe that in $G-T$, 
	there is no $(P^i,P^j)$-path for distinct $i,j\in \{1,2,3\}$.

	\begin{claim}\label{claim:0000x}
	Let $i\in \{1,2,3\}$. No $S$-cycle contains a vertex of $P^i_{mid}$ in $G-T$.
	\end{claim}
	\begin{clproof}
	It is sufficient to prove for $i=1$.	
	For contradiction, suppose that $G-T$ has an $S$-cycle $H$ containing a vertex of $P^1_{mid}$.
	
	As $G-T$ has no $W$-extension, it has no $(P^1_{mid}, W-V(P^1_{mid}))$-path.
	By Lemma~\ref{lem:mid}, 
	 $H$ contains a $(P^1_{v_j}, W-V(P^1_{v_j}))$-path for each $j\in \{1, 2\}$. Call it $X_j$.
	 Then the endpoint of $X_j$ in $W-V(P^1_{v_j})$ is contained in $Q^1\cup Q^2\cup Q^3\cup R^1\cup R^2$.
	By taking a shortest path between the endpoints of $X_1$ and $X_2$ on $P^1$ in $(X_1\cup X_2\cup Q^1\cup Q^2\cup Q^3\cup R^1\cup R^2)- \{v_1, v_2\}$, 
	we can find an $S$-cycle disjoint from $P^2\cup P^3$, which leads a contradiction.
	 	\end{clproof}

\begin{claim}\label{claim:0000x1}
	No $S$-cycle in $G-T$ contains a vertex of $A_{mid}$ for some $A\in \{Q^1,Q^2,R^1,R^2\}$.
	\end{claim}
	\begin{clproof}
	Suppose such an $S$-cycle $H$ exists.
	By symmetry, it is sufficient to consider when $A=Q^1$. 
	As $G$ has no $(Q^1_{mid}, W-V(Q^1_{mid}))$-path, by Lemma~\ref{lem:mid}, 
	 $H$ contains a $(Q^1_{v_1}, W-V(Q^1_{v_1}))$-path, say $B$.
	 If the endpoint of $B$ in $W-V(Q^1_{v_1})$ is contained in $Q^1\cup Q^2\cup Q^3\cup R^1\cup R^2$, 
	 then there is an $S$-cycle disjoint from $P^1\cup P^2$.
	 Thus, we may assume that the endpoint of $B$ is contained in $P^i$ for some $i\in \{1,2,3\}$.
	 If there is a $(P^i, (Q^1\cup Q^2\cup Q^3\cup R^1\cup R^2)-V(Q^1_{v_1}))$-path, 
	then by the same reason, there is an $S$-cycle disjoint from one formed by two other paths of $P^1, P^2, P^3$. 
	
	Let $z$ be the gate of $Q^1$ that is closer to $v_1$.
	Let $I\subseteq \{1, 2, 3\}$ be the set such that for every $j\in I$, there is a $(P^j, Q^1_{v_1})$-path in $G-T$.
	By Claim~\ref{claim:k3+++compo},
	$z$ separates $(\bigcup_{j\in I}P^j) \cup Q^1_{v_1}$ and the rest of $W$ in $G-T$.
	This contradicts that $H-z$ contains a vertex of $Q^1_{v_1}$ and a vertex of $Q^1_{v_4}$ by Lemma~\ref{lem:mid}.
	\end{clproof}
	
	\begin{claim}\label{claim:0000x2}
 No $S$-cycle in $G-T$ contains a vertex of $Q^3_{mid}$.
	\end{claim}
	\begin{clproof}
	Suppose that such an $S$-cycle $H$ exists. 
	Let $C_1$ and $C_2$ be the two connected components of $(Q^1\cup Q^2\cup Q^3\cup R^1\cup R^2)-\{v_1, v_2\}-V(Q^3_{mid})$.
	If there is an $(C_1, C_2)$-path $A$, then there is an $S$-cycle in $Q^1\cup Q^2\cup Q^3\cup R^1\cup R^2\cup A$ 
	disjoint from $P^1\cup P^2$.
	So, there is no such a path.
	Also, for some $i\in \{1, 2, 3\}$, if both a $(C_1, P^i)$-path and a $(C_2, P^i)$-path exist, then 
	there is an $S$-cycle disjoint from an $S$-cycle formed by two other paths in $P^1, P^2, P^3$.
	It implies that for a gate $z$ of $Q^3$ closer to $v_4$, $z$ separates the two parts $C_1$ and $C_2$ in $G-T$.
	This contradicts the assumption that $H-z$ meets both $C_1$ and $C_2$. 
	\end{clproof}
	
	We conclude that $G-T$ has no $S$-cycles.
	\end{proof}

\section{Variations of $W_4$-extensions}\label{sec:varw4}

We now know that $G$ contains an $S$-cycle $H$-subdivision for some 
			$H\in \{W_4, K_{3,3}^+\}$.
In this section, we prove that $G$ contains an $S$-cycle $H$-subdivision for some 
			$H\in \{W_4^+, W_4^*, W_5, K_{3,3}^+\}$, 
			and in the case when $G$ contains an $S$-cycle $H$ subdivision for some 
			$H\in \{W_4^+, W_4^*, W_5\}$, we have 
			$\tau (G, S)\le 4$.

\begin{lemma}\label{lem:W4subdivision}
If $G$ contains an $S$-cycle $W_4$-subdivision, then 
it contains an $S$-cycle $H$-subdivision for some $H\in \{W_4^+, W_4^*, W_5, K_{3,3}^+\}$.
\end{lemma}
\begin{proof}
Let $W$ be an $S$-cycle $W_4$-subdivision in $G$.
	Let $v_1, v_2, v_3, v_4, w$ be the branching vertices of $W$ such that $w$ is the vertex of degree $4$, and
	for each $i\in \{1, 2, 3\}$, let $Q^i$ be the certifying path from $v_i$ to $v_{i+1}$, and
	$Q_4$ be the certifying path from $v_4$ to $v_1$, and
	for each $i\in \{1, 2, 3, 4\}$, $R^i$ be the certifying path from $w$ to $v_i$.

Note that $\abs{E(W_4)}=8$ and $\abs{V(W_4)}=5$. 
	Thus, by Lemma~\ref{lem:smallhitting}, 
	$W$ contains an $S$-cycle hitting set $T\subseteq S\cap V(W)$ such that $\abs{T}\le 4$.
	Since $\tau (G, S)>4$, $G-T$ contains an $S$-cycle, say $C$.
	By Lemma~\ref{lem:intersection2}, 
	either $C$ and $W$ meet at exactly one vertex, 
	or $C$ contains a $W$-extension $X$.
	In the former case, we have two vertex-disjoint $S$-cycles.
	So, we may assume that the latter statement holds.
	
	Assume that the endpoints of $X$ are branching vertices of $W$.
	If one of them is $w$, then $G$ contains an $S$-cycle $W_4^+$-subdivision.
	If they are $(v_1, v_3)$ or $(v_2, v_4)$, then $G$ contains an $S$-cycle $W_4^*$-subdivision.
	Otherwise, $G$ has two vertex-disjoint $S$-cycles, a contradiction.
	
	Next, we assume that exactly one of the endpoints of $X$ is a branching vertex.
	First consider when it is $w$.
	If the other endpoint is in $R^1\cup R^2\cup R^3\cup R^4$, then it creates an $S$-cycle avoiding $Q^1\cup Q^2\cup Q^3\cup Q^4$.
	If the other endpoint is in $Q^1\cup Q^2\cup Q^3\cup Q^4$, then $G$ contains an $S$-cycle $W_5$-subdivision.
	Secondly, we consider when one of $v_1, v_2, v_3, v_4$ is an endpoint of $X$.
	By symmetry, we assume that it is $v_1$.
	If the other endpoint is not in $R^3$, then $G$ contains two vertex-disjoint $S$-cycles.
	If the other endpoint is in $R^3$, then $G$ contains an $S$-cycle $K_{3,3}^+$-subdivision, as required.
	
	Lastly, suppose that both endpoints are not branching vertices.
	If both are contained in $R^1\cup R^2\cup R^3\cup R^4$, then it creates an $S$-cycle avoiding $Q^1\cup Q^2\cup Q^3\cup Q^4$.
	Otherwise, we can find a path between the endpoints in $W-w$ that contains at most two vertices of $\{v_1, v_2, v_3, v_4\}$, 
	and thus we can find an $S$-cycle disjoint from one going through $w$ and two remaining vertices in $\{v_1, v_2, v_3, v_4\}$
	Thus, we have two vertex-disjoint $S$-cycles, a contradiction. 
\end{proof}

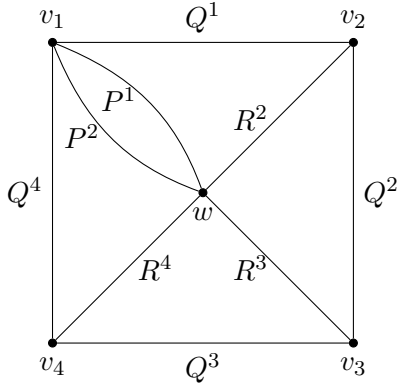
\begin{figure}
 \tikzstyle{v}=[circle, draw, solid, fill=black, inner sep=0pt, minimum width=3pt]
 \tikzstyle{w}=[rectangle, draw, solid, fill=black, inner sep=0pt, minimum width=5pt, minimum height=5pt]
  \centering
   \begin{tikzpicture}[scale=2]

         \node [v, label=below:$v_4$]  (a1) at (0-8, 0-10){};
        \node [v, label=below:$v_3$]  (a2) at (2-8, 0-10){};
        \node [v, label=above:$v_2$]  (a3) at (2-8, 2-10){};
        \node [v, label=above:$v_1$]  (a4) at (0-8, 2-10){};
            \node [v, label=below:$w$]  (a5) at (1-8, 1-10){};
    	\draw(a1) to [edge node={node [below] {$Q^3$}}] (a2);
	\draw(a2) to [edge node={node [right] {$Q^2$}}](a3);
	\draw(a3) to [edge node={node [above] {$Q^1$}}](a4);
	\draw(a4) to [edge node={node [left] {$Q^4$}}](a1);
	\draw(a5) to [edge node={node [right] {$R^4$}}] (a1);    
	\draw(a5) to [edge node={node [left] {$R^3$}}] (a2);    
 	\draw(a5) to [edge node={node [left] {$R^2$}}] (a3);    
            \draw(a5) [in=-22,out=110] to [edge node={node [left] {$P^1$}}]  (a4);
			\draw(a5) [in=-68,out=160] to [edge node={node [left] {$P^2$}}]  (a4);

\end{tikzpicture}
   \caption{The $W_4^+$-subdivision in Proposition~\ref{prop:w4+}.}
  \label{fig:prop:w4+}
\end{figure}

\begin{proposition}\label{prop:w4+}
    If $G$ contains an $S$-cycle $W_4^+$-subdivision, then either it contains an $S$-cycle $K_{3,3}^+$-subdivision or $\tau(G, S)\le 4$.
\end{proposition}
\begin{proof}
    Let $W$ be an S-cycle $W_4^+$-subdivision in $G$. Let $v_1,v_2,v_3,v_4,w$ be the branching vertices of $W$,  let $P^1, P^2$ be the certifying paths from $v_1$ to $w$, $Q^j$ be the certifying path from $v_j$ to $v_{j+1}$ for $j\in \{1,2,3\}$, $Q^4$ be the certifying path from $v_4$ to $v_1$, and $R^k$ be the certifying path from $w$ to $v_k$ for $k\in \{2,3,4\}$. See Figure~\ref{fig:prop:w4+} for an illustration. Suppose that $G$ has no $S$-cycle $K_{3,3}^+$-subdivision.
  
    Let $T$ be the set obtained from $\{v_1,w\}$
    by adding the vertex of $S$ in $(Q^1\cup Q^2)-v_1$ that is closest to $v_3$ if one exists and then 
    adding the vertex of $S$ in $(Q^3\cup Q^4)-v_1$ that is closest to $v_3$ if one exists.
     Note that if $v_3\in S$, then $v_3\in T$.
    
    We observe that $G-\{v_1, w\}$ has no $W$-extension. 
    Indeed, if there is a $W$-extension whose endpoints are in $P^1\cup P^2$, then it creates an $S$-cycle disjoint from $R^3\cup R^4\cup Q^3$, while if there is a $W$-extension whose endpoints are in $W-V(P^1\cup P^2)$, then it creates an $S$-cycle disjoint from $P^1\cup P^2$.
    Assume that one endpoint is in $P^1\cup P^2$ and the other endpoint is not in $P^1\cup P^2$. If the endpoint not in $P^1\cup P^2$ is not $v_3$, then it creates an $S$-cycle 
	disjoint from one of $R^2\cup R^3\cup Q^2, R^3\cup R^4\cup Q^3$ and $Q^1\cup Q^2\cup Q^3\cup Q^4$. If this endpoint is $v_3$, then $G$ has an $S$-cycle $K_{3,3}^+$-subdivision, a contradiction. Thus, 
	 $G-\{v_1, w\}$ has no $W$-extension.
	 Furthermore, $G-\{v_1, w\}$ has no $S$-cycle meeting $W$ on exactly one vertex.
	
    	By applying Lemma~\ref{lem:twocycles}, 
	we can observe that in $G-T$, 
	\begin{itemize}
		\item    there is no $(P^1,P^2)$-path, and
		\item for $i\in \{1,2\}$ and $j\in \{2,4\}$,   there is no $(P^i,R^j-v_j)$-path.
	\end{itemize}

    \begin{claim}\label{mh5}
   Let $i\in \{1,2\}$. No $S$-cycle in $G-T$ contains a vertex of $P^i_{mid}$.
    \end{claim}
    \begin{clproof}
    It is sufficient to show for $i=1$. 
    Suppose that $G-T$ contains an $S$-cycle $H$ containing a vertex of $P^1_{mid}$.
    As $G-T$ has no $W$-extension, 
    $G-T$ has no $(P^1_{mid}, W-V(P^1_{mid}))$-path.
	As $\{v_1, w\}\subseteq T$, by Lemma~\ref{lem:mid}, 
	$H$ contains a $(P^1_v, W-V(P^1_v))$-path for each $v\in \{v_1, w\}$, say $X_v$.
	
	If the endpoint of $X_{v_1}$ in $W-V(P^1_{v_1})$ is not contained in $Q^1\cup Q^2\cup Q^3\cup Q^4$, 
		then it creates an $S$-cycle disjoint from $Q^1\cup Q^2\cup Q^3\cup Q^4$.
		So, the endpoint is contained in $Q^1\cup Q^2\cup Q^3\cup Q^4$.

    If the endpoint of $X_w$ in $W-V(P^1_w)$ is contained in $(Q^1\cup Q^2\cup Q^3\cup Q^4)-v_3$, 
    	then it creates an $S$-cycle vertex-disjoint from one of $R^2\cup R^3\cup Q^2$ and $R^3\cup R^4\cup Q^3$. 
	We observed that there is no $(P^1,P^2)$-path and no $(P^1,R^k-v_k)$-path for $k\in \{2,4\}$ in $G-T$. 
	Thus, the endpoint of $X_w$ in $W-V(P^1_w)$ is contained in $R^3$.
		
		If the endpoint of $X_{v_1}$ is contained in $Q^1\cup Q^2$, then by taking a shortest path between endpoints of $X_{v_1}$ and $X_{w}$ in $Q^1\cup Q^2\cup R^3$, 
		we can find an $S$-cycle disjoint from $P^2\cup R^4\cup Q^4$.
		By symmetry, when the endpoint of $X_{v_1}$ is contained in $Q^3\cup Q^4$, we can find an $S$-cycle disjoint from $P^2\cup R^2\cup Q^1$.
		These are contradictions. Therefore, the claim holds.    
	\end{clproof}
    
    \begin{claim}\label{mh7}
    Let $i\in \{2, 3, 4\}$. No $S$-cycle in $G-T$ contains a vertex of $R^i_{mid}$.
    \end{claim}
    \begin{clproof}
     Suppose that $G-T$ contains an $S$-cycle $H$ containing a vertex of $R^i_{mid}$ for some $i\in \{2,3,4\}$.

	First assume that $i=3$.  
	As $G-T$ has no $W$-extensions, it has no $(R^3_{mid}, W-V(R^3_{mid}))$-path.   
    As $w\in T$, by Lemma~\ref{lem:mid}, 
	$H$ contains an $(R^3_w, W-V(R^3_w))$-path, say $X_w$.
	If the endpoint of $X_w$ in $W-V(R^3_w)$ is not contained in $P^1\cup P^2$, 
	then it creates an $S$-cycle disjoint from $P^1\cup P^2$.
	So, we can assume that the endpoint is contained in $P^1\cup P^2$.
	Without loss of generality, we assume that it is contained in $P^1$.
	If there is also an $(R^3_w, P^2)$-path $B$, then $P^1\cup P^2\cup X_w\cup B\cup R^3_w$ contains an $S$-cycle
	disjoint from one of $R^2\cup R^4\cup Q^2\cup Q^3$ and $Q^1\cup Q^2\cup Q^3\cup Q^4$, a contradiction.
	So, there is no $(R^3_w, P^2)$-path.
	It shows that for every $W$-path in $G-T$ whose one endpoint is in $R^3_w$, 
	the other endpoint is contained in $R^3_w\cup P^1$.
	
	Now, suppose that there is a $W$-path $B$ in $G-T$ whose one endpoint is in $P^1$.
	We know that $G-T$ has no $(P^1, (P^2\cup R^2\cup R^4)-\{v_2, v_4\})$-path.
	If $B$ is a $(P^1, R^3-V(R^3_w)-v_3)$-path, 
	then there is an $S$-cycle disjoint from $Q^1\cup Q^2\cup Q^3\cup Q^4$.
	On the other hand, if an endpoint is contained in $Q^1\cup Q^2\cup Q^3\cup Q^4$, 
	then by taking a shortest path between the endpoints of $X_w$ and $B$ in $Q^1\cup Q^2\cup Q^3\cup Q^4\cup R^3$, 
	we can find an $S$-cycle disjoint from one of $P^2\cup Q^1\cup R^2$ and $P^2\cup R^4\cup Q^4$.
	So, this is not possible. We conclude that the endpoint of $B$ is contained in $P^1\cup R^3_w$.
	
	This implies that the gate of $R^3$ closer to $w$ separates $P^1\cup R^3_w$ from the rest of $W$ in $G-T$.
	This contradicts that there is an $S$-cycle containing a vertex of $R^3_w$ and a vertex of $R^3_{v_3}$.

    Now, we assume that $i\in \{2,4\}$. It is sufficient to show for $R^2$ by symmetry. 
  	Since $G-T$ has no $W$-extension, it has no $(R^2_{mid}, W-V(R^2_{mid}))$-path.
	As $w\in T$, by Lemma~\ref{lem:mid}, 
	$H$ contains an $(R^2_w, W-V(R^2_w))$-path, say $X_w$.
	The endpoint of $X_w$ in $W-V(R^2_w)$ cannot be contained in $P^1\cup P^2$, because there is no $(P^1\cup P^2, R^2-v_2)$-path.
	But otherwise, we can find an $S$-cycle disjoint from $P^1\cup P^2$, a contradiction.
    Thus, we prove the claim.
    \end{clproof}
    
    We prove the last claim.
    
    \begin{claim}\label{mh8}
   No $S$-cycle in $G-T$ contains a vertex of $S$ in $Q^1\cup Q^2\cup Q^3\cup Q^4$.
    \end{claim}
    \begin{clproof}
    Suppose that such a cycle $H$ exists.
    By the definition of $T$, if $v_3\in S$, then $v_3\in T$. 
    So, by symmetry, we may assume that $H$ contains a vertex of $S$ in $(Q^1\cup Q^2)-\{v_1, v_3\}$.
    Note that it is possible that $v_2\in S$ and $H$ contains $v_2$.
    
    Let $u$ be a vertex of $S$ contained in $V(Q^1\cup Q^2)\cap V(H)$.
    First claim that the two neighbors of $u$ in $H$ are neighbors of $u$ in $W$.
    Suppose for contradiction that there is a neighbor $u'$ of $u$ in $H$ that is not a neighbor in $W$.
    As $H$ is a cycle, following the direction from $u$ to $u'$, either we can find a $W$-extension in $G-T$, or $H$ meets exactly $u$ on $W$.
    As $G-T$ has no $W$-extension, $H$ meets exactly $u$ on $W$. But in this case, $H$ is disjoint from $P^2\cup R^4\cup Q^4$, a contradiction.
	Thus, the two neighbors of $u$ in $H$ are neighbors of $u$ in $W$. 
    
    Assume that $u=v_2$ and one neighbor of $u$ in $H$ is contained in $R^2$.
    As $H$ is connected, in $G-\{v_1, w, u\}$, there is a path from $R^2$ to another component of $G-\{v_1, w, u\}$. 
    If the other endpoint is not contained in $P^1\cup P^2$, then one can find an $S$-cycle disjoint from $P^1\cup P^2$, because $v_2\in S$.
    But $G-T$ has no $(R^2, P^1\cup P^2)$-path. So, this is not possible.
    We may assume that when $u=v_2$, the two neighbors of $u$ in $H$ are contained in $Q^1$ and $Q^2$, respectively.
    
    Let $u_1$ be the neighbor of $u$ in $H$ such that $\dist_{Q^1\cup Q^2}(u_1, v_3)$ is minimum.
	Let 
	$a$ be the vertex of $S$ in $(Q^1\cup Q^2)-v_1$ that is closest to $v_3$.
	As $a\in T$ and $H$ does not contain $a$, we have that $a\neq u$ and they are not neighbors in $W$.
	Let $X$ be the connected component of $G-T-u$ containing $u_1$.
	Because $H-u$ is connected, 
	there is a $(X, W-V(X))$-path, say $B$.

	If the endpoint of $B$ in $W-V(X)$ is not contained in $P^1\cup P^2$, 
	then there is an $S$-cycle disjoint from $P^1\cup P^2$ because $a, u\in S$.
	So, we may assume that this endpoint is contained in $P^1\cup P^2$.
	But in this case, $W\cup B$ contains an $S$-cycle disjoint from $R^3\cup R^4\cup Q^3$.
	This is a contradiction.
    \end{clproof}

    By Claims~\ref{mh5}, \ref{mh7}, and \ref{mh8}, $G-T$ has no $S$-cycle containing a vertex of $S$ in $W$.
    Therefore the proposition is true.
\end{proof}

\begin{figure}
 \tikzstyle{v}=[circle, draw, solid, fill=black, inner sep=0pt, minimum width=3pt]
 \tikzstyle{w}=[rectangle, draw, solid, fill=black, inner sep=0pt, minimum width=5pt, minimum height=5pt]
  \centering
   \begin{tikzpicture}[scale=2]

         \node [v, label=below:$w_2$]  (a1) at (0-8, 0-10){};
        \node [v, label=below:$v_3$]  (a2) at (2-8, 0-10){};
        \node [v, label=above:$w_1$]  (a3) at (2-8, 2-10){};
        \node [v, label=above:$v_1$]  (a4) at (0-8, 2-10){};
            \node [v, label=below:$v_2$]  (a5) at (1-8, 1-10){};
    	\draw(a1) to [edge node={node [below] {$Q^3$}}] (a2);
	\draw(a2) to [edge node={node [right] {$P^3$}}](a3);
	\draw(a3) to [edge node={node [above] {$P^1$}}](a4);
	\draw(a4) to [edge node={node [left] {$Q^1$}}](a1);
	\draw(a5) to [edge node={node [right] {$Q^2$}}, near end] (a1);

	\draw(a5) to [edge node={node [right] {$R^2$}}] (a2);    
 	\draw(a5) to [edge node={node [left] {$P^2$}}] (a3);    
			\draw(a4) [in=160,out=-70] to [edge node={node [left] {$R^3$}}, near end]  (a2);

 	\draw(a5) to [edge node={node [right] {$R^1$}}] (a4);    

\end{tikzpicture}
   \caption{The $W_4^*$-subdivision in Proposition~\ref{prop:w4*}.}
  \label{fig:prop:w4*}
\end{figure}
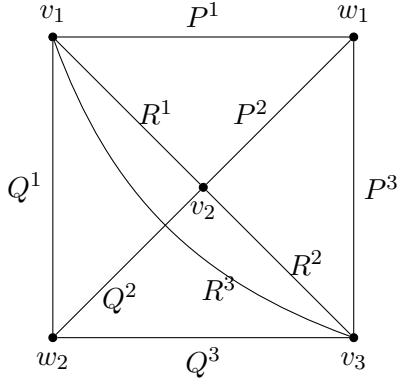

\begin{proposition}\label{prop:w4*}
    If $G$ contains an $S$-cycle $W_4^*$-subdivision, then $\tau(G, S)\le 4$.
\end{proposition}
\begin{proof}
    Let $W$ be an $S$-cycle $W_4^*$-subdivision in $G$. Let $v_1,v_2,v_3,w_1,w_2$ be the branching vertices of $W$, 
    let $P^i$ be the certifying path from $w_1$ to $v_i$ for $i\in \{1,2,3\}$, 
    let $Q^j$ be the certifying path from $w_2$ to $v_j$ for $j\in \{1,2,3\}$, 
    let $R^k$ be the certifying path from $v_k$ to $v_{k+1}$ for $k\in \{1,2\}$, and
    let $R^3$ be the path from $v_3$ to $v_1$. 
    Observe there is a rotational symmetry of $W$ along the cycle $R^1\cup R^2\cup R^3$. 
    See Figure~\ref{fig:prop:w4*} for an illustration. Let $T=\{v_1,v_2,v_3\}$.
    
    	We claim that $G-T$ has no $S$-cycle containing a vertex in $W$.
	Suppose such a cycle $C$ exists.
	By applying Lemma~\ref{lem:twocycles}, 
	we can observe that in $G-T$, 
	there is no $(X^i,R^j)$-path for $X\in \{P, Q\}$ and $i,j\in \{1,2,3\}$.
	If $X^i$ and $R^j$ share an endpoint, then it is easy.
	Suppose $X^i$ and $R^j$ do not share an endpoint; for example, consider $R^1$ and $P^3$.
	If there is an $(R^1, P^3)$-path $Y$ in $G-T$, 
	$Y$ and one of the two subpaths from the endpoint of $Y$ in $R^1$ to $w_1$ in $P^1\cup P^2\cup R^1$  
	and the subpath from $w$ to the endpoint of $Y$ in $P^3$ in $P^3$ form an $S$-cycle disjoint from one of $Q^2\cup Q^3\cup R^2$ and $Q^1\cup Q^3\cup R^3$.

	It implies that $T$ separates $R^1\cup R^2\cup R^3$ and $W-V(R^1\cup R^2\cup R^3)$.
	It further implies that if $C$ contains a vertex of $R^1\cup R^2\cup R^3$, then 
	it is disjoint from $P^1\cup P^2\cup Q^2\cup Q^1$, 
	while if $C$ contains a vertex of $W-V(R^1\cup R^2\cup R^3)$, 
	then it is disjoint from $R^1\cup R^2\cup R^3$.
	Both are contradictions. We conclude that $G-T$ has no $S$-cycle meeting $W$, 
	and $\tau(G, S)\le 3$.
	\end{proof}

\begin{proposition}\label{prop:w5}
	If $G$ contains an $S$-cycle $W_5$-subdivision, then $\tau(G, S)\le 2$.
\end{proposition}
  	\begin{proof}
	Let $W$ be an $S$-cycle $W_5$-subdivision in $G$. 
	Let $v_1, v_2, v_3, v_4, v_5, w$ be the branching vertices of $W$ where $w$ is the vertex of degree $5$ in $W$. 
	For each $i\in \{1,2,3,4\}$, let $Q^i$ be the certifying path from $v_i$ to $v_{i+1}$, and
	let $Q^5$ be the certifying path from $v_5$ to $v_1$, and
	for each $i\in \{1, 2,3, 4, 5\}$, let $R^i$ be the certifying path from $w$ to $v_i$.
	Let $Q:=Q^1\cup \cdots \cup Q^5$. 
	We choose any vertex $q$ of $S$ in $Q$
	and set $T=\{w, q\}$.
	
	We observe that $G-w$ has no $W$-extension. Suppose such a $W$-extension $P$ exists and let $x,y$ be its endpoints.
	In $W-w$, there is a path from $x$ to $y$ containing at most three vertices of $\{v_1, \ldots, v_5\}$.
	Then $P$ with this subpath forms an $S$-cycle disjoint from an $S$-cycle of $W$ going through $w$ and the two remaining vertices of $\{v_1, \ldots, v_5\}$, a contradiction. So, $G-w$ has no $W$-extension.
		 Furthermore, $G-w$ has no $S$-cycle meeting $W$ on exactly one vertex.

	 	By applying Lemma~\ref{lem:twocycles}, 
	we can observe that in $G-T$, 
	there is no $(R^i-v_i, R^j-v_j)$-path for distinct $i,j\in \{1,2,3,4,5\}$.

	\begin{claim}\label{claim:rmid}
	Let $i\in \{1,2,3,4,5\}$.
	No $S$-cycle in $G-T$ contains a vertex in $R^i_{mid}$.
	\end{claim}
	\begin{clproof}
	It suffices to show for $i=1$.
	 Suppose that $G-T$ has an $S$-cycle $H$ containing a vertex of $R^1_{mid}$.
 	As $G-T$ has no $W$-extension,
	it has no $(R^i_{mid}, W-V(R^i_{mid}))$-path.
	So, by Lemma~\ref{lem:mid}, 
	there is an $(R^i_w, W-V(R^i_w))$-path in $G-T$, say $X$. 
	Let $y$ be the endpoint of $X$ such that $y\notin V(R^i_w)$.
	Since $G-T$ has no $(R^a-v_a, R^b-v_b)$-path for distinct $a, b\in \{1,2,3,4,5\}$, 
	$y$ is in $Q^1\cup \cdots \cup Q^5$.
	
	If $y$ is in $Q^1\cup Q^2\cup Q^3-\{v_4\}$, then the $S$-cycle in $R^1\cup Q^1\cup Q^2\cup Q^3\cup X$ is disjoint from the $S$-cycle $R^4\cup R^5\cup Q^4$.
	If $y$ is in $Q^4\cup Q^5$, then the $S$-cycle in $R^1\cup Q^4\cup Q^5 \cup X$ is disjoint from the $S$-cycle $R^2\cup R^3\cup Q^2$.
	We conclude that no $S$-cycle in $G-T$ contains a vertex in $R^i_{mid}$.
	\end{clproof}
	
		\begin{claim}\label{claim:qmid}
	No $S$-cycle in $G-T$ contains a vertex of $S$ in $Q$.
	\end{claim}
	\begin{clproof}
	Suppose that such an $S$-cycle $H$ exists and let $x\in V(H)\cap V(Q)\cap S$.
	As $G-T$ has no $W$-extension and has no $S$-cycle meeting $W$ on one vertex, the neighbors of $x$ in $H$ are contained in $W$, 
	and furthermore, when $x=v_i$ for some $i$, its neighbor in $H$ is not contained in $R^i$.
	Since $q\in T$, $q$ is not a neighbor of $x$ in $W$. 
	Let $x_1$ and $x_2$ be the neighbors of $x$ in $H$, and 
	let $C_1$ and $C_2$ be the connected components of $W-T-x$ containing $x_1$ and $x_2$, respectively. 

	Suppose there exists a $(C_1, W-V(C_1))$-path $X$ and let $y, z$ be the endpoints of $X$ such that  $y\in V(C_1)$. 
	Clearly, any path from $y$ to $z$ in $W-w$  contains a vertex of $S$, because $x, q\in S$.
	It is not difficult to see that there is a path from $y$ to $z$ in $W-w$ contains at most three vertices of $\{v_1, \ldots, v_5\}$.
	Then $X$ and this subpath create an $S$-cycle disjoint from the cycle going through $w$ and two remaining vertices of $W$.
	This is a contradiction.	
	\end{clproof}
	
	We conclude that $G-T$ has no $S$-cycles, and $\tau(G, S)\le 2$.
	\end{proof}

\section{$K_{3,3}^+$-subdivision case}\label{sec:k33final}

We complete the proof of Theorem~\ref{thm:main1} by showing that if $G$ contains an $S$-cycle $K_{3,3}^+$-subdivision, 
then $\tau(G, S)\le 4$.
\begin{proposition}\label{prop:k33+subdivision}
	If $G$ contains an $S$-cycle $K_{3,3}^{+}$-subdivision, then $\tau(G, S)\le 4$.
		\end{proposition}
	\begin{proof}
Let $W$ be an $S$-cycle $K_{3,3}^+$-subdivision in $G$.
	Let $v_1, v_2, v_3, w_1, w_2, w_3$ be the branching vertices of $W$ such that $(\{v_1, v_2, v_3\}, \{w_1, w_2, w_3\})$ corresponds to the bipartition of $K_{3,3}$, 
	and there is also additional certifying path from $v_1$ to $v_2$. 
	For each $i, j\in \{1,2,3\}$, $P^{i,j}$ be the certifying path from $v_i$ to $w_j$,
	and let $Q$ be the additional certifying path from $v_1$ to $v_2$.
	
	Let $T$ be the set obtained from $\{v_1,v_2,v_3\}$
    by adding a gate of $Q$ if $Q_{mid}$ is not empty.
    Let $B=\{v_1, v_2, v_3, w_1, w_2, w_3\}$.
    
    	By applying Lemma~\ref{lem:twocycles}, 
	we can observe that in $G-T$, 
	\begin{itemize}
		\item for $i\in \{1, 2\}$ and distinct $j_1, j_2\in \{1, 2, 3\}$, 
		there is no $(P^{i, j_1}, P^{i, j_2})$-path except when its endpoints are $w_{j_1}$ and $w_{j_2}$, and  
		\item for $i\in \{1, 2\}$ and $j\in \{1, 2, 3\}$, 
		there is no $(P^{i, j}, Q)$-path.
	\end{itemize}

	In the next claim, we show that $G-T$ has no $W$-extension.
	
	\begin{claim}\label{claim:distinctend}
	There is no $W$-extension in $G-T$.
	\end{claim}
	\begin{clproof}
	Note that $W-\{v_1, v_3\}$ is a tree. Suppose $G-T$ has a $W$-extension $X$.
	Let $x,y$ be the endpoints of $X$.
	
	Suppose that $x\in V(Q)$.
	In this case, the unique path from $x$ to $y$ in $W-\{v_1, v_3\}$ uses at most one vertex of $\{w_1, w_2, w_3\}$.
	Let $w_{j_1}, w_{j_2}$ be two vertices not contained in the path from $x$ to $y$ in $W-\{v_1, v_3\}$. 
	Then the union of $X$ and the path from $x$ to $y$ in $W-\{v_1, v_3\}$
	is disjoint from $P^{1, j_1}\cup P^{1, j_2}\cup P^{3, j_1}\cup P^{3, j_2}$, a contradiction.
	So, we may assume that $X$ has no endpoint in $V(Q)$.
	
	Suppose that $x=w_1$.
	If $y=w_j$ for some $j\in \{2,3\}$, 
	then two $S$-cycles $X\cup P^{3, 1}\cup P^{3, j}$ and $Q\cup P^{1, 5-j}\cup P^{2, 5-j}$ are vertex-disjoint, which is a contradiction. 
	So, we may assume that $y\notin \{w_2, w_3\}$.
	Then there is a path from $w_1$ to $y$ in $W$, which contains at most one vertex of $\{v_1, v_2, v_3\}$ and contains no vertex of $\{w_2, w_3\}$.
	It implies that there are two vertex-disjoint $S$-cycles, a contradiction.
	By the same argument, we may assume that any of $w_1, w_2, w_3$ is not an endpoint of $X$.	
	
	Now assume that the two endpoints of $X$ are contained in $(\bigcup_{i, j\in \{1,2,3\}} V(P^{i,j})) \setminus B$.
	In case when the two certifying paths containing $x$ and $y$ share an endpoint, then it is easy to see that there are two vertex-disjoint $S$-cycles.
	We assume that the two paths, say $P^{i_1, j_1}$ and $P^{i_2, j_2}$, containing $x$ and $y$ respectively, do not share an endpoint.
	Let $i_3\in \{1,2,3\}\setminus \{i_1, i_2\}$ and $j_3\in \{1,2,3\}\setminus \{j_1, j_2\}$.
    In this case, the $S$-cycle in $P^{i_1, j_1}\cup P^{i_2, j_2}\cup P^{i_2, j_1}\cup X$ is disjoint from 
    the $S$-cycle $P^{i_1, j_2}\cup P^{i_1, j_3}\cup P^{i_3, j_2}\cup P^{i_3, j_3}$, which leads a contradiction.
	\end{clproof}
	Also, $G-T$ has no $S$-cycle meeting $W$ on exactly one vertex.

	We will show that $G-T$ has no $S$-cycle containing a vertex in $S\cap V(W)$.
	If this is true, then by Lemma~\ref{lem:conclusion}, $T$ is an $S$-cycle hitting set and thus $\tau (G, S)\le 4$.
	By the choice of $T$ and Lemma~\ref{lem:mid}, no $S$-cycle contains a vertex of $Q_{mid}$.

	\begin{claim}
	Let $i\in \{1,2\}$ and $j\in \{1,2,3\}$.
	No $S$-cycle in $G-T$ contains a vertex of $S$ in $P^{i,j}$. 
	\end{claim}
	\begin{clproof}
	First we show that no $S$-cycle in $G-T$ contains a vertex in $P^{i, j}_{mid}$.
	By symmetry, it is sufficient to show for $i=j=1$.
			Suppose for contradiction that there is such an $S$-cycle.
			As every $(P^{1,1}_{mid}, W-V(P^{1,1}_{mid}))$-path in $G-T$ is a $W$-extension,
			by Claim~\ref{claim:distinctend}, 
			there is no $(P^{1,1}_{mid}, W-V(P^{1,1}_{mid}))$-path in $G-T$.		
				So, by Lemma~\ref{lem:mid}, 
			  there is a $(P^{1,1}_{v_1}, W-V(P^{1,1}_{v_1}))$-path, say $Y$.
			  Let $x$ and $y$ be the endpoints of $Y$ such that $x\in V(P^{1,1}_{v_1})$.
			  We observed that $y$ cannot be in $P^{1,2}\cup P^{1,3}\cup Q$, and it cannot be in $V(P^{1,1})\setminus V(P^{1,1}_{v_1})$.
			   We analyze the remaining cases.
			  
			  If $y\in V(P^{t, 1})$ for some $t\in \{2,3\}$, then $Y$ with the subpath of $P^{1,1}\cup P^{t,1}$ from $x$ to $y$
			  forms an $S$-cycle disjoint from $P^{1,2}\cup P^{1,3}\cup P^{5-t,2}\cup P^{5-t,3}$.
			If $y\in V(P^{p,q})\setminus B$ for some $p,q\in \{2,3\}$,
			then the $S$-cycle in $Y\cup P^{1,1}\cup P^{p,1}\cup P^{p,q}$ is disjoint from 
				the $S$-cycle $P^{1,2}\cup P^{1,3}\cup P^{5-p,2}\cup P^{5-p,3}$. 
				So, both cases are not possible.
				Thus, there is no $(P^{1,1}_{v_1}, W-V(P^{1,1}_{v_1}))$-path, a contradiction.
				We conclude that no $S$-cycle in $G-T$ contains a vertex in $P^{i, j}_{mid}$ for all $i\in \{1,2\}$ and $j\in \{1,2,3\}$.
				
				Now, we assume that $w_k\in S$ and there is an $S$-cycle $H$ in $G-T$ containing $w_k$ for some $k\in \{1,2,3\}$.
				As $G-T$ has no $W$-extension, the two neighbors of $w_k$ in $H$ are contained in $W$.
				So, one of the neighbors of $w_k$ in $H$ is contained in 
				$V(P^{\ell,k})\setminus B$ for some $\ell \in \{1,2\}$. 
				Then by the above argument, we can show that 
				such an $S$-cycle does not exist.
			\end{clproof}
			
			\begin{claim}
			Let $i\in \{1,2,3\}$. No $S$-cycle in $G-T$ contains a vertex of $P^{3, i}_{mid}$.
			\end{claim}
			\begin{clproof}
			It suffices to show for $i=1$.
			Suppose for contradiction that such an $S$-cycle $H$ exists, 
			and let $z\in V(H)\cap V(P^{3,i}_{mid})\cap S$. 
			As $G-T$ has no $W$-extension,  
			the two neighbors of $z$ in $H$ are the neighbors in $W$.
			Let $z_1$ and $z_2$ be the two neighbors of $z$ such that $\dist_{P^{3,i}}(v_3, z_1)\le \dist_{P^{3,i}}(v_3, z_2)$, 
			and let $C_1$ and $C_2$ be the two components of $W-T-z$ containing $z_1$ and $z_2$, respectively.
			Since $H-z$ is a path, there is a $(C_2, W-V(C_2))$-path in $G-T$, say $Y$.
	 	Let $y_1$ and $y_2$ be the endpoints of $Y$ such that $y_1\in V(C_2)$.
			
			Assume that $y_2$ is contained in $C_1\cup (\bigcup_{i\in \{1,2,3\}, j\in \{2,3\}}P^{i,j})$.
			Then the path from $y_1$ to $y_2$ in $W-\{v_1, v_2\}$ and $Y$ form an $S$-cycle that is vertex-disjoint from 
			one of $Q\cup P^{1,2}\cup P^{2,2}$ and $Q\cup P^{1,3}\cup P^{2,3}$.
			This is a contradiction. Thus, we may assume that $y_2$ is contained in $Q$.
			
			Observe that there are two paths from $y_1$ to $y_2$ in $Q\cup P^{1,1}\cup P^{2,1}\cup P^{3,1}$, where one of them must contain a vertex of $S$. Thus, $Y$ and one of the two paths form an $S$-cycle, 
			and this $S$-cycle is disjoint from 
			one of $P^{1,2}\cup P^{1,3}\cup P^{3,2}\cup P^{3,3}$ and 
			    $P^{2,2}\cup P^{2,3}\cup P^{3,2}\cup P^{3,3}$. This is a contradiction.

			This proves the claim.
			\end{clproof}
				
			We conclude that $G-T$ has no $S$-cycles, as required.
\end{proof}

\begin{proof}[Proof of Theorem~\ref{thm:main1}]
Suppose that $\mu(G, S)\le 1$ and $\tau(G, S)>4$.

By Lemma~\ref{lem:reduction1}
	$G$ contains an $S$-cycle $H$-subdivision for some 
	$H\in \{K_3^{+++}, K_4\}$, 
	and by Proposition~\ref{prop:k3triple}, 
	if $G$ contains an $S$-cycle $K_3^{+++}$-subdivision, then $\tau(G,S)\le 3$, a contradiction.
	So, $G$ contains an $S$-cycle $K_4$-subdivision.
Then by Lemma~\ref{lem:reduction2}
	$G$ contains an $S$-cycle $H$-subdivision for some 
			$H\in \{K_4^{++}, K_4^{+++}, W_4, K_{3,3}^+\}$.
			By Propositions~\ref{prop:k4++} and \ref{prop:k4+++}, 
	if $G$ contains an $S$-cycle $H$-subdivision for some 
			$H\in \{K_4^{++}, K_4^{+++}\}$, 
			then $G$ contains an $S$-cycle $K_{3,3}^+$-subdivision.
	So, 	$G$ contains an $S$-cycle $H$-subdivision for some 
			$H\in \{W_4, K_{3,3}^+\}$.

	By Lemma~\ref{lem:W4subdivision}, 
if $G$ contains an $S$-cycle $W_4$-subdivision, then 
it contains an $S$-cycle $H$-subdivision for some $H\in \{W_4^+, W_4^*, W_5, K_{3,3}^+\}$.
When $G$ contains an $S$-cycle $H$-subdivision for some $H\in \{W_4^+, W_4^*, W_5\}$
we have $\tau(G, S)\le 4$ by Propositions~\ref{prop:w4+}, \ref{prop:w4*}, and \ref{prop:w5}.
	Thus, $G$ contains an $S$-cycle $K_{3,3}^+$-subdivision, and in this case, $\tau(G, S)\le 4$ by Proposition~\ref{prop:k33+subdivision}.
	This is a contradiction.
	
We conclude that $\tau(G, S)\le 4$.
\end{proof}

\section{Concluding notes}
We prove that if a rooted graph $(G, S)$ has no two vertex-disjoint $S$-cycles, then $\tau(G, S)\le 4$, and 
this bound cannot be improved to $3$. A natural question is to determine the tight bound when $(G, S)$ has no three vertex-disjoint $S$-cycles.

\begin{question}
What is the minimum integer $c$ such that every rooted graph $(G, S)$ with $\mu(G, S)\le 2$ satisfies $\tau(G, S)\le c$?
\end{question}

\providecommand{\bysame}{\leavevmode\hbox to3em{\hrulefill}\thinspace}
\providecommand{\MR}{\relax\ifhmode\unskip\space\fi MR }
\providecommand{\MRhref}[2]{%
  \href{http://www.ams.org/mathscinet-getitem?mr=#1}{#2}
}
\providecommand{\href}[2]{#2}

\end{document}